\documentclass[reqno]{amsart}

\usepackage{xypic,hyperref,todonotes,comment}
\usepackage{graphics,amssymb,multicol}
\usepackage{pagecolor,lipsum}
\usepackage{latexsym}
\usepackage{mathrsfs}
\xyoption{curve}
\tikzstyle{arrow} = [thick,->,>=latex]
\usetikzlibrary{shapes,arrows}
\usepackage[normalem]{ulem}

\definecolor{color1}{HTML}{6E001B}
\definecolor{color2}{HTML}{39E600}
\definecolor{pagecolor}{HTML}{FF8484}

\usepackage{hyperref}
\hypersetup{
colorlinks,
citecolor=color1,
filecolor=blue,
linkcolor=color1,
urlcolor=color1
}

\newcommand{\nocontentsline}[3]{}
\newcommand{\tocless}[2]{\bgroup\let\addcontentsline=\nocontentsline#1{#2}\egroup}

\newtheorem{lemma}{Lemma}[section]
\newtheorem{proposition}[lemma]{Proposition}
\newtheorem{theorem}[lemma]{Theorem}
\newtheorem{corollary}[lemma]{Corollary}
\newtheorem{question}[lemma]{Question}

\newtheorem*{theoremA}{Theorem}

\theoremstyle{definition}
\newtheorem{example}[lemma]{Example}
\newtheorem{definition}[lemma]{Definition}
\newtheorem{remark}[lemma]{Remark}

\newcommand{\mfk}[1]{\mathfrak{#1}}
\newcommand{\mbb}[1]{\mathbb{#1}}
\newcommand{\mcl}[1]{\mathcal{#1}}
\newcommand{\mrm}[1]{\mathrm{#1}}
\newcommand{\msc}[1]{\mathscr{#1}}

\newcommand{\msf}[1]{\mathsf{#1}}

\DeclareMathOperator{\Hom}{Hom}
\DeclareMathOperator{\End}{End}
\DeclareMathOperator{\RHom}{RHom}

\DeclareMathOperator{\Ext}{Ext}

\DeclareMathOperator{\Aut}{Aut}
\DeclareMathOperator{\rep}{rep}
\DeclareMathOperator{\Rep}{Rep}
\DeclareMathOperator{\Spec}{Spec}

\DeclareMathOperator{\Proj}{Proj}

\DeclareMathOperator{\Coh}{Coh}
\DeclareMathOperator{\QCoh}{QCoh}
\DeclareMathOperator{\Supp}{Supp}
\DeclareMathOperator{\ord}{ord}

\DeclareMathOperator{\res}{res}
\DeclareMathOperator{\supp}{supp}
\DeclareMathOperator{\Sym}{Sym}
\DeclareMathOperator{\Ind}{Ind}

\let\oldO\O

\newcommand{\Sing}{\operatorname{Sing}}

\newcommand{\stab}{\operatorname{stab}}
\newcommand{\Stab}{\operatorname{Stab}}
\newcommand{\cSpec}{\mathsf{Spec}}

\newcommand{\ot}{\otimes}
\renewcommand{\1}{\mathbf{1}}
\renewcommand{\O}{\mathscr{O}}
\renewcommand{\hat}{\widehat}
\renewcommand{\tilde}{\widetilde}

\title[]{Hypersurface support and prime ideal spectra for stable categories}
\date{\today}

\author{Cris Negron}
\address{Department of Mathematics, University of Southern California, Los Angeles, CA 90007}
\email{cnegron@usc.edu}

\author{Julia Pevtsova}
\address{Department of Mathematics, University of Washington, Seattle, WA 98195}
\email{julia@math.washington.edu}

\begin{document}
\maketitle

\begin{abstract}
We use hypersurface support to classify thick (two-sided) ideals in the stable categories of representations for several families of finite-dimensional integrable Hopf algebras: bosonized quantum complete intersections, quantum Borels in type $A$, Drinfeld doubles of height $1$ Borels in finite characteristic, and rings of functions on finite group schemes over a perfect field.  We then identify the prime ideal (Balmer) spectra for these stable categories.  In the curious case of functions on a finite group scheme $\mcl{G}$, the spectrum of the category is identified not with the spectrum of cohomology, but with the quotient of the spectrum of cohomology by the adjoint action of the subgroup of connected components $\pi_0(\mcl{G})$ in $\mcl{G}$.
\end{abstract}

\section{Introduction}

The problem of classifying thick tensor ideals in a given tensor triangulated category takes its roots in stable homotopy theory thanks to the groundbreaking work of Devinatz-Hopkins-Smith \cite{devinatz-hopkins-smith88, hopkins-smith98}. It was subsequently taken up to the derived categories of rings and schemes in \cite{hopkins87, neeman92, thomason97} and to the stable category of a finite group in \cite{bensoncarlsonrickard97}, eventually leading to the elegant framework of computing the spectrum of a tensor triangulated category introduced by Balmer \cite{balmer05} and now considered as part of the tensor triangular (or tt-) geometry with numerous classification results coming from all realms of algebraic, topological and geometric contexts.   In this work, we study tensor triangular geometry of stable categories of representations $\stab(\msf{u})$ for certain classes of finite-dimensional {\it integrable} Hopf algebras $\msf{u}$ via hypersurface support.
\par

For the modular representation theory of a finite group (and a finite group scheme) $G$ the classification of thick tensor ideals (equivalently, calculation of the Balmer spectrum) in $\stab(\msf{u})$, where $\msf{u} =kG$ is the group algebra, is achieved by identifying the cohomological support with Carlson's rank variety or the $\pi$-support. This approach uses the cocommutativity of $\msf u$ in an essential way and appears to be very hard to extend to the quantum situation we are interested in, beyond some very special choice of parameters for a quantum linear plane as was done in (\cite{pevtsovawitherspoon15}, see also \cite{pevtsovawitherspoon09, bensonerdmannholloway07}). 

The hypersurface support, a concept introduced for complete intersections in commutative algebra in \cite{eisenbud80, avramovbuchweitz00, avramoviyengar18}, proved to be a more versatile construction.  In \cite{negronpevtsova} we developed hypersurface support for noncommutative complete intersections with an eye towards calculating the Balmer spectrum for $\stab(\msf{u})$ as a main application. 

For such calculations one frequently needs to understand support not only for  $\stab(\msf{u})$, but also for the big stable category $\Stab(\msf{u})$ of {\it all} $\msf{u}$-repesentations in which $\stab(\msf{u})$ resides as the subcategory of rigid compact objects.  The necessary extension of hypersurface support from $\stab(\msf{u})$ to the big stable category $\Stab(\msf{u})$ is provided in \cite{negronpevtsova2}, where its compatibilities with the \emph{triangulated} structure were considered.  Here we analyze compatibilities between hypersurface support and the \emph{monoidal} structure on $\Stab(\msf{u})$, see Definition \ref{def:text} and Section \ref{sect:tensor}, specifically Theorem~\ref{thm:tpp} which is a version of the tensor product property for supports in our situation.
\par

A key feature of hypersurface support is that it is developed for stable categories which are not necessarily symmetric or even braided, and so it provides certain ``data" regarding the behaviors of support for non-braided categories.  One of the original motivations for the present study was to pursue a categorical framing of support, directly via thick tensor ideals, which anticipates and explains the types of phenomena we observe at the levels of cohomological and hypersurface support, e.g.\ \cite[Theorem 10.8, 11.6]{negronpevtsova}.  In this work we argue that notions of \emph{central generation} for tensor ideals connect categorical and cohomological approaches to support in the non-braided context.  These central generation properties are automatically satisfied in the braided setting and provide an explicit link between braided/symmetric studies of tensor triangular geometry and their non-braided counterparts.  We elaborate on these points more below.
\vspace{0.1in}

Let us now discuss the contents of the paper in more detail.  Fix $k$ an arbitrary field of any characteristic.  We recall, from \cite{negronpevtsova}, that a finite-dimensional Hopf algebra is called \emph{integrable} if it admits a smooth deformation $U\to \msf{u}$ by a Noetherian Hopf algebra $U$ which is of finite global dimension.  We require in this case that the parametrizing subalgebra $Z\subset U$ is a central Hopf subalgebra, or more generally coideal subalgebra, in $U$.  This framework covers and unifies many interesting families of Hopf algebras: restricted enveloping algebras of Lie algebras in positive characteristic, small quantum groups and their Borels in characteristic zero, Drinfeld doubles of infinitesimal group schemes of height one in positive characteristic, and commutative complete intersections, including the classical representation theoretic case of the group algebra of an elementary abelian $p$-group over a field of characteristic $p$. 

We recall also that a support theory $(Y,\supp)$ for a tensor triangulated category $\mcl{T}$ is an assignment of a closed subspace $\supp(V)$ in a topological space $Y$ to each $V$ in $\mcl{T}$ which properly acknowledges both the triangulated structure and tensor structure on $\mcl{T}$ (see Sections \ref{sect:supports} and \ref{sect:stick} below).
For an integrable Hopf algebra $\msf{u}$, we use the integration $U\to \msf{u}$ to produce a support theory $(\msf{Y},\supp^{hyp}_\mbb{P})$ for the stable category $\stab(\msf{u})$ of $\msf{u}$-representations, which is referred to as \emph{hypersurface support} \cite{negronpevtsova,negronpevtsova2}.  The target space $\msf{Y}$ for hypersurface support is (generally a finite quotient of) the projective spectrum of cohomology $\Proj\Ext^\ast_{\msf{u}}(k,k)$.  The precise construction of the support theory $(\msf{Y},\supp^{hyp}_\mbb{P})$ for $\stab(\msf{u})$, and its cocompletion $\Stab(\msf{u})$, is recalled in Section \ref{sect:hyp}.
\par

This paper essentially has two halves.  In the first portion of the paper we present a general framing of support, and tt-geometry, for non-braided tensor triangular categories.  The emphasis here is on the manner in which various centralizing hypotheses can be applied to obtain classifications of thick ideals from suitably well-behaved support theories.  In the second portion of the paper we apply this generic framing to show that, for a number of explicit families of such integrable $\msf{u}$, hypersurface support can be used to classify thick ideals in the associated stable category of representations $\stab(\msf{u})$. We also calculate the spectrum $\cSpec(\stab(\msf{u}))$ of prime ideals in $\stab(\msf{u})$ for the given families.  These calculations of the spectrum are new so that, via these classes of examples, we are testing this general theory in a rather direct and substantial way.
\par

For us, the spectrum $\cSpec(\stab(\msf{u}))$ is defined in direct analogy with Balmer's spectrum of a symmetric tensor category \cite{balmer05} (see Section \ref{sect:spec}).  Also, our strategy for classifying thick ideals is motivated by the original approach of Benson-Carlson-Rickard \cite{bensoncarlsonrickard97} (see Section \ref{sect:abstract}).
\par

The families of integrable Hopf algebras which we study in this text all satisfy a Chevalley property (Definition~\ref{eq:337}), which reflects the fact that their categories of representations are formally similar to, say, representations for a Frobenius kernel in a (quantum) Borel. Such ``Borel-like" categories form the foundations of most support theoretic studies in representation theory, see, for example, \cite{quillen71,suslinfriedlanderbendel97}, and also Section \ref{sect:non-braided}.

We provide the proposed classification of thick ideals, and calculate the prime ideal spectrum, for the stable representation categories of the following families of integrable Hopf algebras:
\begin{enumerate}
\item[(F1)]\label{it:1} Quantum complete intersections, aka quantum linear spaces.
\item[(F2)]\label{it:2} Small quantum Borels in type $A$.
\item[(F3)]\label{it:3} Drinfeld doubles $\mcl{D}(B_{(1)})$ for Borel subgroups $B\subset \mbb{G}$ in almost-simple algebraic groups.
\item[(G)]\label{it:G} Algebras of functions $\O(\mcl{G})$ on arbitrary finite group schemes $\mcl{G}$, over a perfect field.
\end{enumerate}
These families are described in more detail in Section \ref{sect:examples}.  We note that our results for the quantum Borel hold at an \emph{arbitrary} odd order parameter $q$, not just $q$ greater than the corresponding Coxeter number, and that our results for the double $\mcl{D}(B_{(1)})$ only require that the characteristic $p$ is very good for the associated Dynkin type. The following theorem summarizes our results for the first three families.  

\begin{theoremA}[\ref{thm:qci_ideals}, \ref{thm:borel_ideals}, \ref{thm:db_ideals}]
For a Hopf algebra $\msf u$ belonging to one of the families {\rm (\hyperref[it:1]{F1})--(\hyperref[it:3]{F3})}, cohomological support $(\Proj \Ext^*_\msf{u}(k,k), \supp^{coh})$ classifies thick ideals in $\stab(\msf u)$. There is furthermore a homeomorphism 
\[
\Proj \Ext^*_\msf{u}(k, k) \overset{\cong}\longrightarrow \cSpec(\stab(\msf u)). 
\]
\end{theoremA}

The case (\hyperref[it:G]{G}) of functions on a finite group scheme $\mcl{G}$ is unique among the examples considered.  However, the behaviors of support for $\Coh(\mcl{G})=\rep(\O(\mcl{G}))$ may be more representative of phenomena for supports of finite tensor categories in general, when compared with (\hyperref[it:1]{F1})--(\hyperref[it:3]{F3}).  By $\Coh(\mcl{G})$ here we mean the generally non-symmetric tensor category of sheaves on $\mcl{G}$ with tensor structure induced by the group structure on $\mcl{G}$, or rather induced by the Hopf structure on $\O(\mcl{G})$.  We provide the following classification result.

\begin{theoremA}[\ref{thm:G_classification}]
Consider $\msf{u}=\O(\mcl{G})$, for $\mcl{G}$ a finite group scheme over a perfect field, and let $\pi=\mcl{G}_{\rm red}$ be the reduced subgroup in $\mcl{G}$.  Thick ideals in $\stab(\msf{u})$ are classified by a support theory which takes values in the quotient $\left(\Proj\Ext^*_{\msf{u}}(k, k)\right)/\pi$ of the spectrum of cohomology by the adjoint action of $\pi$.  We furthermore have a homeomorphism
\[
\left(\Proj\Ext^*_{\msf{u}}(k, k)\right)/\pi\overset{\cong}\longrightarrow \cSpec(\stab(\msf{u})).
\]
\end{theoremA}

As mentioned above, our analyses of the spectra $\cSpec(\stab(\msf{u}))$ focus on notions of {central generation} for ideals in the stable category.  When all ideals in $\stab(\msf{u})$ are generated by (sufficiently) central objects, one can essentially employ ``commutative" arguments in order to understand the universal support theory $(\cSpec(\stab(\msf{u})),\supp^{uni})$ for $\stab(\msf{u})$.  One can see for example Question \ref{q:quest_gen}, Proposition \ref{prop:properties}, Theorem \ref{thm:balmer}, Theorem \ref{thm:center_genI}, and the materials of Section \ref{sect:cen_G}.  Our ability to analyze non-braided categories via central objects also explains, at least to a certain degree, why geometries for non-braided tensor categories behave more like geometries associated to commutative, rather than noncommutative, rings.  One can compare our approach with the ``noncommutative tensor triangulated geometry" proposed in \cite{nakanovashawyakimov,nakanovashawyakimovII}.
\par

A nontrivial amount of space in the exposition is dedicated to discussions of thick ideals, and prime ideals, in stable categories of arbitrary non-braided tensor categories (see in particular Sections \ref{sect:ideals} and \ref{sect:spec}.)  We have found these explorations to be quite interesting in their own right, but one can, and arguably should, connect the aspirations of the present text with those of more familiar/popular studies of support theory for braided tensor categories.  So, we take a moment to discuss this point.

\subsection{Why non-braided categories?}\label{sect:non-braided}

Depending on an individual's motivations, they may ``only" be interested in braided tensor categories.  We sympathize with this perspective.  However, let us make a point regarding supports for tensor categories outside of the classical (symmetric) setting.  Below by \emph{support} we generally mean \emph{cohomological support} (see Section \ref{sect:supports}).
\par

In finite characteristic, one may approach support theory for representations of a given finite group scheme $\mcl{G}$ by studying the supports of objects over its collection of unipotent subgroups $\mcl{U}_\mu\subset \mcl{G}$.  Rather, one studies the support of an object $V$ in $\rep(\mcl{G})$ by considering its restrictions along the various symmetric tensor functors $\res_\mu:\rep(\mcl{G})\to \rep(\mcl{U}_\mu)$, then essentially glues the supports of $V$ over the $\mcl{U}_\mu$ together to recover information about the support of $V$ over the global object $\mcl{G}$ \cite{suslinfriedlanderbendel97,friedlanderpevtsova07}.  The point here is that, in general, support over a given unipotent group $\mcl{U}_\mu$ is, in some ways, more tractable than that of $\mcl{G}$.
\par

In a braided but non-symmetric setting, it can be useful to take a similar approach.  Given a braided tensor category $\msc{Z}$ over an arbitrary field, we would like to study the support of an object $V$ in $\msc{Z}$ by studying the supports of its images along a distinguished collection of surjective tensor functors $\{F_\mu\}_{\mu\in \mrm{M}}$.  These functors will be maps to some more manageable categories $\msc{Z}_\mu$,
\[
F_\mu:\msc{Z}\to \msc{Z}_\mu.
\]
(So we stratify, in a sense, $\msc{Z}$ by the functors $F_\mu$, cf.\ \cite[Corollary 9.16]{mathew16}.)  In characteristic $0$--and probably in most non-symmetric cases in finite characteristic--it is \emph{not} reasonable to require the $\msc{Z}_\mu$ here to be braided.  For example, for representations of the small quantum group $\msc{Z}=\rep(u_q(\mbb{G}))$, associated to an almost-simple algebraic group $\mbb{G}$, it is natural to take the $\msc{Z}_\mu$ to be the representation categories of the varied quantum Borels $u_q(B_\mu)$ for $\mbb{G}$ at $q$.  The categories $\rep(u_q(B_\mu))$ are known to admit no braidings \cite{bonteanikshych}.  This is explicitly the perspective taken in the upcoming work \cite{negronpevtsovaII}.
\par

It therefore becomes natural, or even necessary, to provide refined analyses of support for non-braided tensor categories.  This is true even when one is primarily interested in braided categories.

\subsection{Structure of the paper} 

Sections \ref{sect:ftc}--\ref{sect:recoll} are introductory.  In Sections~\ref{sect:ideals} and \ref{sect:abstract} we provide a general discussion of thick ideals, localizing subcategories and support theories for (stabilized) finite tensor categories $\msc{C}$.  In Section~\ref{sect:ideals} we introduce the notion of a centrally generated thick ideal, which turns out to be essential for the rest of the paper. In the same section we list axiomatic properties for a {\it multiplicative} support theory, and a {\it tensor extension} of such a theory to the big stable category (Definitions~\ref{def:reasonable} and \ref{def:text}).  We explain how such a support theory classifies thick ideals in the stable category $\stab(\msc C)$. In Section~\ref{sect:spec}, following Balmer \cite{balmer05}, we discuss the universal support theory for $\stab(\msc C)$ and the associated spectrum of prime ideals. We show in Theorem~\ref{thm:balmer} that, under certain central generation assumptions, a multiplicative support theory for $\stab(\msc C)$ which admits a tensor extension to $\Stab(\msc{C})$ can be used to calculate the spectrum $\cSpec(\stab(\msc C))$, up to homeomorphism.
\par 

We turn to a (re)consideration of hypersurface support in Section~\ref{sect:tensor}. Our goal is to show that hypersurface support provides a tensor extension of cohomological support for the classes of integrable Hopf algebras (\hyperref[it:1]{F1})--(\hyperref[it:3]{F3}) discussed above.  In Section~\ref{sect:tensor} we show that the appropriate extended tensor product property for hypersurface support follows from a ``Thick subcategory lemma" for hypersurface algebras (whose prototype, also known as the Hopkins lemma, first appeared in \cite{hopkins87}).  In Section~\ref{sect:ideals_q} we apply the thick subcategory lemma to classify thick tensor ideals and compute the spectrum for quantum complete intersections, small quantum Borels in type $A$, and the Drinfeld doubles $\mcl{D}(B_{(1)})$.  
\par

Finally, in Sections \ref{sect:hyper_G} and \ref{sect:ideals_G} we analyze the case (\hyperref[it:G]{G}) of the stable category $\stab(\Coh(\mcl{G}))$ of coherent sheaves on a finite group scheme $\mcl{G}$.  In the last section, Section~\ref{sect:onesided}, we briefly discuss one-sided versus two-sided thick ideals.

\subsection{Acknowledgements}

Thanks to Lucho Avramov, Eric Friedlander, Henning Krause, Dan Nakano, and Mark Walker for helpful commentary.  Special thanks to Srikanth Iyengar for offering myriad insights throughout the production of this work.  The first named author is supported by NSF grant DMS-2001608.  The second named author is supported by NSF grants DMS-1901854 and the Brian and Tiffinie Pang faculty fellowship.  This material is based upon work supported by the National Science Foundation under Grant No.\ DMS-1440140, while the first author was in residence at the Mathematical Sciences Research Institute in Berkeley, California, during the Spring 2020 semester, and the second author was in digital residence.

\tableofcontents

\section{Finite tensor categories etc.}
\label{sect:ftc}

Throughout this work $k$ is a field of arbitrary characteristic.  In this section we recall some basic information about finite tensor categories.  Recall, say from \cite{bakalovkirillov01} or \cite{egno15}, that a tensor category (over $k$) is a $k$-linear, abelian rigid monoidal category which has a simple unit object $\1$, finite-dimensional Hom sets, and all objects of finite length. Following \cite{etingofostrik04}, a tensor category $\msc{C}$ is called finite if it has finitely many simples and enough projectives.  We call $\msc{C}$ a fusion category if it is finite and semisimple.  All tensor functors are exact, by definition, and tensor subcategories are full, by definition.
\par

Our main examples of (finite) tensor categories are representation categories $\msc{C}=\rep(\msf{u})$ of finite-dimensional Hopf algebras, with monoidal structure induced by the coproduct on $\msf{u}$.  There are, however, many examples of finite tensor categories which are not representation categories of Hopf algebras (see for example \cite{tsuchiyawood13,gelakisebbag,bensonetingof}).

\subsection{Centralizers}
\label{sect:centralizers}

Consider $\msc{C}$ a finite tensor category and $\msc{D}\subset \msc{C}$ a tensor subcategory. To such a pair $(\msc{D},\msc{C})$ we can associate the \emph{Drinfeld centralizer} of $\msc{D}$ in $\msc{C}$ \cite{muger03,shimizu19II}.  This is the (finite) tensor category consisting of pairs $(V,\gamma_V)$, where $V$ is an object in $\msc{C}$ and $\gamma_V:V\ot -\to -\ot V$ is a natural isomorphism between the functors $V\ot-,-\ot V:\msc{D}\to \msc{C}$ which satisfies the constraint
\[
(id_X\ot \gamma_{V,Y})(\gamma_{V,X}\ot id_Y)=\gamma_{V,X\ot Y},
\]
at all $X$ and $Y$ in $\msc{D}$.  We call $\gamma_V$ a $\msc{D}$-centralizing structure on $V$.  We denote the the category of $\msc{D}$-centralizing objects in $\msc{C}$ by
\[
Z^\msc{D}(\msc{C}):=\left\{\begin{array}{c}\text{the category of pairs $(V,\gamma_V)$ of}\\
\text{an object $V$ in $\msc{C}$ and a choice of}\\
\text{$\msc{D}$-centralizing structure $\gamma_V$ on $V$}
\end{array}\right\}.
\]
Morphisms $f:(V,\gamma_V)\to (W,\gamma_W)$ in $Z^{\msc{D}}(\msc{C})$ are those maps $f:V\to W$ in $\msc{C}$ which respect the centralizing structures, in the sense that $(-\ot f)\gamma_V=\gamma_W(f\ot -)$.
\par

In the extreme case $\msc{D}=\msc{C}$, we obtain the Drinfeld center $Z(\msc{C})$ of $\msc{C}$.  This is the category of pairs $(V,\gamma_V)$ of objects in $\msc{C}$ with global central structures.  We have the forgetful functors $Z(\msc{C})\to \msc{C}$ and $Z^\msc{D}(\msc{C})\to \msc{C}$.
\par

In this text we are interested in two specific cases; the case where $\msc{D}=\msc{C}$ and the case where $\msc{D}$ is the tensor subcategory generated by the semisimple objects in $\msc{C}$.  We say, informally, that an object $V$ in $\msc{C}$ ``is central" if $V$ admits a lift to the Drinfeld center $Z(\msc{C})$.  Similarly, in the case where $\msc{D}$ is generated by the simples, we say an object $V$ ``centralizes the simples" if $V$ admits a lift to $Z^{\msc{D}}(\msc{C})$.

\begin{remark}
We've suppressed associators in the above formulas.  Also, by the tensor subcategory generated by a class of objects $\{X_i\}_i$ we mean the smallest full tensor subcategory which contains the $X_i$ and is closed under taking subquotients.
\end{remark}

\begin{remark}
Practically speaking, we consider the relative centralizer $Z^{\msc{D}}(\msc{C})$ only in the case in which the tensor subcategory $\msc{D}$ generated by the semisimple objects is just the subcategory of semisimple objects itself.  This is a kind of solvability condition on the category $\msc{C}$.  One can compare, for example, with the case of representations of small quantum $\mfk{sl}_2$ at a root of unity, in which case the subcategory generated by the $2$-dimensional simple is all of $\rep(u_q(\mfk{sl}_2))$.  (See also Section \ref{sect:geom_chev}.)
\end{remark}

\subsection{Stable categories for finite tensor categories}
\label{sect:stabC}

Recall that any finite tensor category $\msc{C}$ is Frobenius \cite[Proposition 2.3]{etingofostrik04}.  So we can define the associated stable category $\stab(\msc{C})=\msc{C}/\operatorname{proj}(\msc{C})$.  This is, more precisely, the additive category with the same objects as $\msc{C}$, and morphisms given as the quotient of morphisms in $\msc{C}$ by all those maps which factor through a projective.  The category $\stab(\msc{C})$ is naturally tensor triangulated \cite{happel90,rickard89}, i.e. triangulated with a compatible rigid monoidal structure.  The monoidal structure on $\stab(\msc{C})$ is induced directly by that of $\msc{C}$.
\par

In the stable setting we still speak of central objects, and objects which centralize the simples.  We employ this language in the same manner as outlined in the previous subsection.  For example, for $\msc{D}\subset \msc{C}$ the full tensor subcategory generated by the simples, the forgetful functor $F:Z^\msc{D}(\msc{C})\to \msc{C}$ is surjective (aka dominant), and therefore sends projectives in $Z^\msc{D}(\msc{C})$ to projectives in $\msc{C}$ \cite{etingofostrik04}.  Hence we can stabilize this map
\[
\stab(F):\stab(Z^\msc{D}(\msc{C}))\to \stab(\msc{C}),
\]
and an object in $\stab(\msc{C})$ is said to centralize the simples if it admits a lift along $\stab(F)$.

We are also interested in the ``big" stable category $\Stab(\msc{C})$, which is defined as the stabilization of the Ind-category $\Stab(\msc{C})=\operatorname{Ind}(\msc{C})/\Proj(\operatorname{Ind}(\msc{C}))$.  This category is monoidal and triangulated, but not rigid, and it is compactly generated with compact objects $\Stab(\msc{C})^c=\stab(\msc{C})$.

\begin{remark}
If one expresses $\msc{C}$ as the representation category $\msc{C}\cong \rep(A)$ of a Hopf algebroid $A$--which can always be done \cite{szlachanyi00}--then one recovers the expected formula $\Stab(\msc{C})\cong\Rep(A)/\Proj(A)$.
\end{remark}

\section{Recollections from \cite{negronpevtsova,negronpevtsova2}}
\label{sect:recoll}

We recall some essential information which was covered in \cite{negronpevtsova}.  Here we employ basic notions from deformation theory.  By a deformation of an algebra $R$ we mean a choice of an augmented commutative algebra $Z$, a flat $Z$-algebra $Q$, and a map $Q\to R$ which reduces to an isomorphism $k\ot_ZQ\cong R$ at the augmentation $1:Z\to k$.  In this case we call $Z$ the parametrizing algebra, or the parametrization algebra, of the deformation.

\subsection{Integrable Hopf algebras}

\begin{definition}[\cite{negronpevtsova}]\label{def:integrable}
A finite-dimensional Hopf algebra $\msf{u}$ is said to be \emph{smoothly integrable}, or just {\it integrable}, if $\msf{u}$ admits a deformation $U\to \msf{u}$ parametrized by a smooth central subalgebra $\mcl{Z}\subset U$ such that
\begin{enumerate}
\item[(a)] $U$ is a Noetherian Hopf algebra of finite global dimension, and $U\to \msf{u}$ is a Hopf map.
\item[(b)] $\mcl{Z}$ is a coideal subalgebra in $U$.
\end{enumerate}
Under these conditions, $U$ is called a (smooth) integration of $\msf{u}$ parametrized by $\mcl{Z}$.  An integration $U\to \msf{u}$ is called conormal if $\mcl{Z}$ is a Hopf subalgebra in $U$.
\end{definition}

We suppose specifically that $\mcl{Z}$ is a \emph{right} coideal subalgebra in $U$, so that the comultiplication on $U$ restricts to a coaction $\Delta:\mcl{Z}\to \mcl{Z}\ot U$.  This provides, for any $f\in m_Z$, an action of $\rep(\msf{u})$ on the \emph{left} of the associated hypersurface category $U/(f)\text{-mod}$.  Also, smoothness of $\mcl{Z}$ implies, in particular, that $\mcl{Z}$ is of finite type over $k$.

\begin{remark}
We generally reserve the notation $\rep(B)$ for the \emph{tensor} category of representations for a Hopf algebra $B$, while $B$-mod denotes the abelian category of modules for a generic $k$-algebra.
\end{remark}

Throughout we employ the notations
\[
\Lambda=\msf{u}/\operatorname{Jac}(\msf{u})\ \text{ and }\ m_{\mcl{Z}}\subset \mcl{Z}
\]
for the sum of the simple representations for $\msf{u}$, with multiplicity, and the distinguished maximal ideal in $\mcl{Z}$, respectively.  Rather, $m_{\mcl{Z}}\subset \mcl{Z}$ is the kernel of the augmentation $1:\mcl{Z}\to k$ at which we have $k\ot_\mcl{Z}U\cong \msf{u}$.

Given an integration $U\to \msf{u}$ of some finite-dimensional Hopf algebra $\msf{u}$, we can complete at the distinguished point $1\in \Spec(\mcl{Z})$ to produce a formal counterpart
\[
\msf{U}\to \msf{u},\ \ \text{where $\msf{U}$ is the completion}\ \msf{U}:=Z\ot_{\mcl{Z}}U,\ Z=\hat{\mcl{Z}}_1=\varprojlim_n\mcl{Z}/m_{\mcl{Z}}^n.
\]
The completed algebra $\msf{U}$ is a Hopf algebra in the category of linear topological vector spaces, and $Z\subset \msf{U}$ is a coideal subalgebra in $\msf{U}$.  Furthermore, $\msf{U}$ is Noetherian and of finite global dimension \cite[Lemma 2.10]{negronpevtsova}.
\par

One can find in \cite[\S 2.1]{negronpevtsova} a gaggle of examples of integrable Hopf algebras.  In the present work, as in \cite{negronpevtsova2}, we always assume that a formal integration $\msf{U}\to \msf{u}$ has a smooth analog $U\to \msf{u}$.  So, although we work almost exclusively with the formal pair $(Z,\msf{U})$, we assume the existence of a smooth predecessor $(\mcl{Z},U)$.

\subsection{Hypersurface support}
\label{sect:hyp}

Consider integrable $\msf{u}$, with fixed (formal) integration $\msf{U}\to \msf{u}$, and parametrizing algebra $Z$.  We have the maximal ideal $m_Z\subset Z$, and for an arbitrary point in the projective space $\mbb{P}(m_Z/m_Z^2)$,
\[
c:\Spec(K)\to \mbb{P}(m_Z/m_Z^2),
\]
we consider an associated hypersurface algebra $\msf{U}_c=\msf{U}_K/(f)$, where $f\in m_{Z_K}$ is any element such that the class $\bar{f}\in (m_Z/m_Z^2)_K$ is a representative for $c$.  Note that for such $\msf{U}_c$ the base changed integration $\msf{U}_K\to \msf{u}_K$ descends to a map $\msf{U}_c\to \msf{u}_K$ which realizes the hypersurface algebra as a deformation of $\msf{u}_K$.

\begin{remark}
By $Z_K$ and $\msf{U}_K$ we mean the topological base change $Z_K=\varprojlim_n (Z/m_Z^n)\ot K$ and $\msf{U}_K=\varprojlim_n (\msf{U}/m_Z^n\msf{U})\ot K$, respectively.
\end{remark}

For any $\msf{u}$-module $M$, either finite-dimensional or infinite-dimensional, we say that $M$ is supported at such a point $c$ if the base change $M_K$ is of infinite projective dimension over $\msf{U}_c$, and we define the hypersurface support of $M$ as
\[
\supp^{hyp}_\mbb{P}(M):=\left\{
\begin{array}{c}
\text{the image of all points }c:\Spec(K)\to \mbb{P}(m_Z/m_Z^2)\\
\text{at which  $\operatorname{projdim}_{\msf{U}_c}(M_K) = \infty$} 
\end{array}\right\}.
\]
We show in \cite{negronpevtsova2} that there is no ambiguity in this definition.  For example, the projective dimension of $M_K$ over $\msf{U}_c$ is independent of the choice of representative $f\in m_Z$.  We also have the following detection theorem. 

\begin{theorem}[{\cite[Theorem 6.1]{negronpevtsova2}}]\label{thm:detec}
For $M$ any $\msf{u}$-module, $\supp^{hyp}_\mbb{P}(M)=\emptyset$ if and only if $M$ is projective over $\msf{u}$.  Equivalently, $\supp^{hyp}_\mbb{P}(M)=\emptyset$ if and only if $M$ vanishes in the stable category $\Stab(\msf{u})$.
\end{theorem}

By deformation theory \cite{gerstenhaber64} \cite[Theorem 1.1.5]{bezrukavnikovginzburg07}, the integration $\msf{U}\to \msf{u}$ furthermore provides a graded algebra map
\[
A_Z:=\Sym(\Sigma^{-2}(m_Z/m_Z^2)^\ast)\to \Ext^\ast_\msf{u}(k,k)
\]
from the homogeneous coordinate ring of $\mbb{P}(m_Z/m_Z^2)$, with generators in degree $2$.  This map can be shown to be finite \cite[Theorem 4.8]{negronpevtsova}.  Thus we have, dually, a closed map of schemes
\[
\kappa:\msf{Y}\to \mbb{P}(m_Z/m_Z^2)
\]
from the projective spectrum of cohomology $\msf{Y}:=\Proj\Ext^\ast_\msf{u}(k,k)_{\rm red}$ to this projective space.  In \cite{negronpevtsova} we prove the following.

\begin{proposition}[{\cite[Theorem 7.1]{negronpevtsova}}]\label{prop:538}
Consider $\msf{u}$ an integrable Hopf algebra, with fixed integration $\msf{U}\to \msf{u}$.  Take $V$ in $\rep(\msf{u})$ and let $\Ext^\ast_\msf{u}(V,V)^\sim$ denote the (coherent) sheaf on $\msf{Y}$ defined by the tensor action $V\ot-$ of $\Ext^\ast_{\msf{u}}(k,k)$ on $\Ext^\ast_\msf{u}(V,V)$.  We have
\[
\supp^{hyp}_\mbb{P}(V)=\kappa\big(\Supp_{\msf{Y}}\Ext^\ast_\msf{u}(V,V)^\sim\big).
\]
In particular, hypersurface support vanishes on the open complement $\mbb{P}-\kappa(\msf{Y})$.
\end{proposition}

We provide a further analysis of hypersurface support for certain integrable Hopf algebras in Section \ref{sect:tensor} below.

\subsection{Geometrically Chevalley algebras \cite[\S 8.3, 8.4]{negronpevtsova}}
\label{sect:geom_chev}

In the following definition we let $\Rep(\Lambda)$ denote the monoidal category of infinite-dimensional representations for a finite-dimensional Hopf algebra $\Lambda$.  Recall that, for an integration $\mcl{Z}\subset U\to \msf{u}$ of an integrable Hopf algebra $\msf{u}$, we take $Z=\varprojlim_n \mcl{Z}/m_\mcl{Z}^n$.

\begin{definition}[\cite{negronpevtsova}]\label{eq:337}
Call an integrable Hopf algebra $\msf{u}$ geometrically Chevalley if the following hold:
\begin{enumerate}
\item[(a)] $\msf{u}$ is the bosonization $\msf{u}=\msf{u}^+\rtimes \Lambda$ of a local Hopf algebra $\msf{u}^+$ in a (semisimple!) braided fusion category $\rep(\Lambda)$.
\item[(b)] $\msf{u}^+$ admits a deformation $\mcl{Z}\subset U^+\to \msf{u}^+$ via algebras in $\Rep(\Lambda)$ such that 
\begin{enumerate}
\item[(b0)] $\msf{U}^+$ is a Hopf algebra in $\Rep(\Lambda)$, and $\msf{U}^+\to \msf{u}^+$ is a Hopf map.
\item[(b1)] $\mcl{Z}$ is a central Hopf subalgebra in $U^+$ which has trivial $\Lambda$-action, and is smooth as a commutative $k$-algebra.
\item[(b2)] the associated completion $\msf{U}^+$ is a local Hopf algebra in $\Rep(\Lambda)$ which is of finite global dimension, as an associative algebra.
\end{enumerate}
\end{enumerate}
We call the smash product $U=U^+\rtimes \Lambda\to \msf{u}$ the corresponding Chevalley integration for $\msf{u}$.  The integration $U\to \msf{u}$ is also parametrized by the Hopf subalgebra $\mcl{Z}$.
\end{definition}

Note that for any geometrically Chevalley $\msf{u}$ the augmentation on $\msf{u}^+$ induces a tensor embedding $\rep(\Lambda)\to \rep(\msf{u})$.  This embedding identifies $\rep(\Lambda)$ with the fusion subcategory of semisimple objects in $\rep(\msf{u})$.  In particular, the subcategory of semisimple objects in $\rep(\msf{u})$ is seen to be a \emph{tensor} subcategory.  The importance of this class of Hopf algebras, at least as far as this study is concerned, lies in the following lemma.

\begin{lemma}[{\cite[Lemma 8.2]{negronpevtsova}}]\label{lem:w/e}
If $\msf{u}$ is geometrically Chevalley, then all objects in $\rep(\msf{u})$ centralize the simples.  More specifically, the forgetful functor $Z^{\rep(\Lambda)}(\rep(\msf{u}))\to \rep(\msf{u})$ admits a canonical tensor section $\rep(\msf{u})\to Z^{\rep(\Lambda)}(\rep(\msf{u}))$.
\end{lemma}

The two families of geometrically Chevalley Hopf algebras which we consider here are Drinfeld doubles of height 1 Borels, and the (small) quantum Borel.  We recall these examples below.

\subsection{Some examples}
\label{sect:examples}

We list some families of Hopf algebras which are investigated in detail throughout the text.  We note that the results of this work are not \emph{exclusive} to these families.  Rather, we intend to provide a (relatively) diverse collection of examples which highlight various aspects of hypersurface support, centralizing hypotheses on objects in tensor categories, spectra of stable categories, etc.
\par

We consider the following families of Hopf algebras, which also appeared in \cite{negronpevtsova}:
\begin{enumerate}
\item[(F1)]\label{it:F1} Bosonized quantum complete intersections (aka quantum linear spaces) $\msf{a}_q=\msf{a}_q(P)$, with integration $A_q\to \msf{a}_q$.
\item[(F2)]\label{it:F2} Small quantum Borels $u_q(B)$ in type $A$, at arbitrary odd order $q$, with integration $U^{DK}_q(B)\to u_q(B)$ provided by the De Concini-Kac algebra.
\item[(F3)]\label{it:F3} Drinfeld doubles $\mcl{D}(B_{(1)})$ for $B$ a Borel in an almost simple algebraic group $\mbb{G}$, over $\overline{\mbb{F}}_p$ at arbitrary $p$.  We integrate such $\mcl{D}(B_{(1)})$ via the Hopf algebra $(\O(B)\rtimes U(\mfk{n}))\rtimes kT_{(1)}$.  Here $\mfk{n}$ is the nilpotent radical in the Lie algebra for $B$, and $T\subset B$ is the maximal torus.
\item[(F4)]\label{it:F4} Functions $\O(\mcl{G})$ on a connected finite (aka infinitesimal) group sheme, with integration $\O(\mcl{H})\to \O(\mcl{G})$ provided by a choice of embedding $\mcl{G}\to \mcl{H}$ into a smooth, connected, algebraic group $\mcl{H}$.
\end{enumerate}

We recall our specific constructions of quantum complete intersections and quantum Borels below.  For (\hyperref[it:F4]{F4}) the deformation $\O(\mcl{H})$ is parametrized by functions on the associated quotient $\mcl{Z}=\O(\mcl{H}/\mcl{G})$, and for (\hyperref[it:F3]{F3}) $\mcl{Z}$ is the product $\mcl{Z}=\O(B^{(1)})\ot Z_0(\mfk{b})$ of functions on the quotient $B/B_{(1)}\cong B^{(1)}$ with the Zassenhaus subalgebra $Z_0(\mfk{b})=k[x^{[p]}-x^p:x\in \mfk{b}]$.
\par

With respect to (\hyperref[it:F4]{F4}), we are also interested in functions $\O(\mcl{G})$ on \emph{non-}connected group schemes $\mcl{G}$.  Such algebras, however, cannot be treated in the same manner as (\hyperref[it:F1]{F1})--(\hyperref[it:F4]{F4}), and are afforded their own analysis.

\begin{remark}\label{rem:comm}
The example (\hyperref[it:F4]{F4}) of functions on a connected group scheme can be understood via preexisting results for commutative local rings.  In particular, thick ideals in the stable category for $\O(\mcl{G})$ are classified by Stevenson \cite{stevenson13}.  We consider the example here for two reasons: first, it is an instructive example to keep in mind, and second, an analysis of this case is necessary for our analysis of $\O(\mcl{G})$ at non-connected $\mcl{G}$.
\end{remark}

\subsection{Elaborations}

The algebra $\msf{a}_q(P)$ is specifically the smash product $\msf{a}_q^+(P)\rtimes G$ of a finite group with a truncated skew polynomial ring
\begin{equation}\label{eq:174}
\msf{a}^+_q=\msf{a}^+_q(P)=\mbb{C}\langle x_1,\dots, x_n\rangle/(x_ix_j-q_{ij}x_jx_i,\ x_i^l)_{i\neq j}.
\end{equation}
Here $q$ is a root of unity of odd order $l$, $P=[a_{ij}]$ is an integer matrix for which $q_{ij}=q^{a_{ij}}$, $a_{ij}=-a_{ji}$ off the diagonal, and $a_{ii}=1$.  The group $G$ is the finite abelian group $(\mbb{Z}/l\mbb{Z})^n$ with generators $K_i$ such that $K_i\cdot x_j=q_{ij}x_i$.  We define the positive algebra $A_q^+$ by omitting the nilpotence relations in \eqref{eq:174}, and take $A_q=A_q^+\rtimes G$.  The parametrizing subalgebra $\mcl{Z}$ for the deformation $A_q\to \msf{a}_q$ in this case is the subalgebra $\mbb{C}[x_1^l,\dots,x_n^l]\subset A_q$ generated by the $l$-th powers of the $x_i$.

Let us also recall our construction of the (small) quantum Borel \cite{negron,negronpevtsova}.  We follow the presentation of \cite[\S 9]{negron}, which allows for certain synergies between small and large quantum groups, via quantum Frobenius.
\par

Consider $\mfk{g}$ a simple Lie algebra over $\mbb{C}$, and let $X$ be an intermediate lattice $Q\subset X\subset P$ between the associated root and weight lattices.  Rather, choose an almost-simple algebraic group $\mbb{G}$ which is of the same Dynkin type as $\mfk{g}$.  We consider the $q$-exponentiated normalized Killing form $q^{(-,-)}$ on $X$, where $q$ is again of (finite) odd order $l$.  The form $(-,-)$ is normalized so that the length of short roots is $2$, and we deal with fractional values of $(-,-)$ by formally choosing some root $\sqrt[r]{q}$ of $q$.
\par

Let $X^M\subset X$ denote the radical of the exponentiated form $q^{(-,-)}$, and consider the induced non-degenerate form on the quotient $X/X^M$.  We then have the associated braided fusion category
\[
Vect_{X/X^M},\ \ \text{the fusion category of vector spaces graded by }X/X^M.
\]
To be explicit, for $\nu,\mu\in X/X^M$ the braiding on this category is given by $c_{\mu,\nu}:\mbb{C}_\mu\ot\mbb{C}_\nu\to \mbb{C}_\nu\ot\mbb{C}_\mu$, $c_{\mu,\nu}(1_\mu,1_\nu)=q^{(\mu,\nu)}1_\nu\ot 1_\mu$.  We may take now $\Lambda=\mbb{C}G$ to be the group algebra of the characters $G:=(X/X^M)^\vee$ to rewrite this category $Vect_{X/X^M}$ as representations of a semisimple, quasitriangular, Hopf algebra $\Lambda$, in accordance with Definition \ref{eq:337}.
\par

The positive subalgebra $u_q(\mfk{n})$ in the small quantum group $u_q(\mfk{g})$ \cite{lusztig90,lusztig90II} is a Hopf algebra in the braided fusion category $Vect_{X/X^M}$, with each generator $E_\alpha$ of degree $\alpha$, and we define the quantum Borel for $\mbb{G}$ at $q$ as the bosonization
\[
u_q(B)=u_q(\mfk{n})\rtimes G.
\]
The quantum Borel is geometrically Chevalley, with Chevalley integration provided by the de Concini-Kac form $U_q^{DK}(\mfk{n})\to u_q(\mfk{n})$,
\[
U_q^{DK}(B)=U_q^{DK}(\mfk{n})\rtimes G\to u_q(B).
\]
The parametrizing subalgebra $\mcl{Z}$ in this case is generated by the $l$-th powers of the root vectors $E_\gamma\in U_q^{DK}(B)$, where $\gamma$ runs over all positive roots \cite{deconcinikac91}.

\section{Thick ideals and support for stable tensor categories}
\label{sect:ideals}

We begin our discussion of thick ideals and supports for stable categories.  Throughout the section we consider a finite tensor category $\msc{C}$, and its corresponding stable category $\stab(\msc{C})$.
\par

One of the points of this portion of the paper is to explain how various centralization hypotheses can be employed to reduce analyses of thick ideals in the (generally non-braided) category $\stab(\msc{C})$ to ones which are essentially aligned with the braided/commutative analyses of, say, \cite{balmer05,friedlanderpevtsova07}.  We also discuss localization functors, and at the conclusion of the section we provide some remarks concerning abelian versus triangulated categories.

\subsection{Thick and localizing ideals in stable categories}
\label{sub:thick}
One should recall at this point the necessary notions of central, and centralizing, objects from Section \ref{sect:ftc}.  We recall also that a thick subcategory in a triangulated category is a full triangulated subcategory which is closed under taking summands.

\begin{definition}
A thick ideal in $\stab(\msc{C})$ is a thick subcategory which is preserved under the actions of $\stab(\msc{C})$ on the left and right.
\par

A thick ideal $\msc{I}$ is said to be centrally generated if $\msc{I}$ is generated by a collection of objects which are central in $\stab(\msc{C})$ (in the sense of \ref{sect:stabC}).
\end{definition}

By the ideal generated by a collection of objects $\{V_i\}_{i\in I}$ we mean the smallest thick ideal $\langle V_i:i\in I\rangle^{\ot}$ in $\stab(\msc{C})$ which contains all of the generators $V_i$.  

In addition to considering thick subcategories and thick ideas, we also consider localizing subcategories and localizing ideals in the big stable category $\Stab(\msc{C})$.  By a localizing subcategory (resp.\ ideal) we mean a thick subcategory (resp.\ ideal) which is additionally closed under arbitrary set-indexed sums.  For a thick ideal $\msc{I}$ in $\stab(\msc{C})$, the localizing subcategory $\msc{I}_{\rm loc}$ which it generates in $\Stab(\msc{C})$ is a localizing ideal \cite[Lemma 5.8]{rickard97}.

\subsection{An aside on centralities and methods}

Before getting into any technicalities, let us take a moment to orient ourselves around this notion of a centrally generated ideal, and its significance to our study.  We view the central generation property for thick ideals as fundamental, although it only makes a minimal appearance in this section.  Centrally generated ideals also appear to be the ``good ideals" in $\stab(\msc{C})$, from the perspective of classification, and are the ``important ideals" from the perspective laid out in Section \ref{sect:non-braided}.
\par

In all of the examples we are able to handle explicitly, \emph{all} ideals in $\stab(\msc{C})$ are shown to be centrally generated (see Sections \ref{sect:central} and \ref{sect:cen_G}).  We present the following question, which we return to in Section \ref{sect:cen_G}.

\begin{question}\label{q:quest_gen}
For any finite tensor category $\msc{C}$, are all thick ideals in $\stab(\msc{C})$ centrally generated?
\end{question}

Of course, this question has to do with the relationship between $\msc{C}$ and its Drinfeld center $Z(\msc{C})$, and also to do with the relationship between their cohomologies (cf.\ \cite{negronplavnik}).  One sees from varied studies in the subject that the nature of $Z(\msc{C})$ as a braided tensor category, and the nature of $\msc{C}$ as a ``plain" tensor category, are strongly intertwined, see e.g.\ \cite{ngschauenburg07,etingofnikshychostrik10,changcui19}.
\par

Although we are primarily interested in central generation of ideals, there are additional weaker notion of centrality and weaker notions of central generation which have proved useful to consider.  Specifically, we have found it (very) useful to consider ideals in $\stab(\msc{C})$ which are generated by objects which centralize the simples in $\msc{C}$.
\par

In general we use this weaker notion of centrality in our classification of thick two-sided ideals in the stable category $\stab(\msc{C})$, for a given tensor category $\msc{C}$.  One can see for example Proposition \ref{prop:classify} and Theorem \ref{thm:balmer} below.  Given such a classification we then conclude that all ideals in $\stab(\msc{C})$ are in fact centrally generated, as in Theorem \ref{thm:center_genI} below.  Now, let us continue.

\subsection{Supports for stable categories}
\label{sect:supports}

By a support theory $(Y,\supp)$ for the stable category $\stab(\msc{C})$ we mean a choice of a topological space $Y$, and an assignment $V\mapsto\supp(V)$ of a closed subset in $Y$ to each object $V$ in $\stab(\msc{C})$, which respects the triangulated structure on the stable category.   Specifically, we demand the following:
\begin{itemize}
\item $\supp(V)=\supp(\Sigma V)$ for all $V$ in $\msc{C}$, and $\supp(0)=\emptyset$.
\item $\supp(V\oplus W)=\supp(V)\cup\supp(W)$.
\item Any exact triangle $V\to W\to V'\to \Sigma V$ implies an inclusion of supports $\supp(W)\subset (\supp(V)\cup\supp(W))$.
\end{itemize}

Of course, our main examples are hypersurface support and cohomological support.  To recall, cohomological support is defined as follows: The extensions of the unit $\Ext^\ast_\msc{C}(\1,\1)$ act on $\Ext^\ast_\msc{C}(V,V)$ via the tensor structure $V\ot -$, and for 
\[
\msf{Y}=\Proj\Ext^\ast_\msc{C}(\1,\1)_{\rm red}
\]
we consider the associated sheaf $\Ext^\ast_\msc{C}(V,V)^\sim$ on $\msf{Y}$ and define
\[
\supp^{coh}_\msf{Y}(V):=\operatorname{Supp}_\msf{Y}\Ext^\ast_\msc{C}(V,V)^\sim.
\]
We note that hypersurface and cohomological supports are often identified, in the cases where they are both defined, via Proposition \ref{prop:538} and \cite[Proposition 3.3(iii)]{berghplavnikwitherspoon}.

By a \emph{triangular} extension of a particular theory $(Y,\supp)$ to the big stable category $\Stab(\msc{C})$ we mean an assignment of sub\emph{sets} $\supp(M)$ in $Y$ to objects $M$ in $\Stab(\msc{C})$ which respects the triangulated structure, just as in the compact case.  However, we require additionally that the extended support splits over arbitrary sums in $\Stab(\msc{C})$, $\supp(\oplus_{i\in I}M_i)=\cup_{i\in I}\supp(M_i)$.

\subsection{Supports and thick ideals}
\label{sect:stick}

In this subsection we describe various types of support theories which arise in practice.  In this text we are only interested in support theories which take values in Noetherian schemes.  So, we always take this point for granted.

\begin{definition}\label{def:reasonable}
Consider a support theory $(Y,\supp)$ for $\stab(\msc{C})$.  Assume that $Y$ is homeomorphic to a Noetherian scheme.  We call $(Y,\supp)$
\begin{itemize}
\item \emph{faithful} if vanishing of support $\supp(V)=\emptyset$ implies $V\cong 0$.
\item \emph{exhaustive} if all closed subsets in $Y$ are realized as supports of objects in $\stab(\msc{C})$.
\item \emph{multiplicative} if $\supp(V\ot W)\subset \big(\supp(V)\cap\supp(W)\big)$ for all $V$ and $W$, and
\[
\supp(V\ot W)=\supp(V)\cap \supp(W)
\]
whenever one of $V$ or $W$ centralizes the simples.
\end{itemize}
\end{definition}

Cohomological support theories are popular because they are faithful and exhaustive \cite{berghplavnikwitherspoon}, but they are not multiplicative in general, at least when $\msc{C}$ is non-braided.  One can consider, for example, the category of sheaves $\Coh(\mcl{G})$ on a finite non-connected group scheme.  As was observed in \cite{bensonwitherspoon14,pevtsovawitherspoon15} \cite[Example 10.2]{negronpevtsova}, for particular choices of such $\mcl{G}$, and particular choices of $V$ and $W$ in $\Coh(\mcl{G})$, the inclusion
\[
\supp^{coh}_\msf{Y}(V\ot W)\nsubseteq \big(\supp^{coh}_\msf{Y}(V)\cap\supp^{coh}_\msf{Y}(W)\big)
\]
will already fail.

\begin{remark}
When $\msc{C}$ is braided, all objects centralize the simples, so that the multiplicative property requires that $\supp(V\ot W)=\supp(V)\cap\supp(W)$ at all $V$ and $W$.  We wouldn't claim that the above notion of multiplicativity is definitively \emph{the} correct notion outside of the braided setting, but it is sufficiently functional for the moment.  One has the obvious notion of a sub-multiplicative theory, which we leave the reader to ponder at their leisure.
\end{remark}

Recall that a specialization closed subset in a topological space $Y$ is a subset $\Theta$ which contains the closure $\overline{x}\subset \Theta$ of each point $x\in \Theta$.  Consider now $(Y,\supp)$ a multiplicative support theory for $\stab(\msc{C})$.  Then for any specialization closed subset $\Theta\subset Y$ the associated subcategory of objects which are supported in $\Theta$,
\[
\msc{I}_\Theta:=\{\text{The full subcategory of objects $V$ in $\stab(\msc{C})$ with }\supp(V)\subset \Theta\},
\]
is a thick ideal in $\stab(\msc{C})$.  Conversely, for any thick ideal $\msc{I}$ in $\stab(\msc{C})$ the subset
\[
\supp(\msc{I}):=\cup_{V\in\msc{I}}\supp(V)
\]
is specialization closed in $Y$.  So we have maps
\begin{equation}\label{eq:1058}
\{\text{\rm thick ideals in }\stab(\msc{C})\}\underset{\supp(?)}{\overset{\msc{I}_?}\leftrightarrows}
\{\text{\rm specialization closed subsets in }Y\}
\end{equation}
which restrict to maps
\begin{equation}\label{eq:1063}
\{\text{\rm finitely generated ideals in }\stab(\msc{C})\}\underset{\supp(?)}{\overset{\msc{I}_?}\leftrightarrows} \{\text{\rm closed subsets in }Y\}.
\end{equation}

\begin{definition}
A multiplicative support theory $(Y,\supp)$ for $\stab(\msc{C})$ is said to \emph{classify thick ideals} in $\stab(\msc{C})$ if the maps \eqref{eq:1058} and \eqref{eq:1063} are mutually inverse bijections.
\end{definition}

The classification of ideals in $\stab(\msc{C})$ is strongly related to the existence of well-behaved extensions of support theories for $\stab(\msc{C})$ to the big stable category.  We consider the following

\begin{definition}\label{def:text}
A \emph{tensor extension} of a support theory $(Y,\supp)$ for $\stab(\msc{C})$ is a triangular extension to $\Stab(\msc{C})$ such that
\begin{enumerate}
\item[(a)] for any $V$ in $\stab(\msc{C})$ and $M$ in $\Stab(\msc{C})$ we have a containment
\[
\supp(V\ot M)\subset \big(\supp(M)\cap\supp(V)\big)\ \ \text{\bf or}\ \ \supp(M\ot V)\subset \big(\supp(M)\cap\supp(V)\big).
\]
\item[(b)] for any $V$ in $\stab(\msc{C})$ which centralizes the simples, and arbitrary $M$ in $\Stab(\msc{C})$, we have an equality
\[
\supp(V\ot M)=\supp(V)\cap \supp(M)\ \ \text{\bf or}\ \ \supp(M\ot V)=\supp(M)\cap\supp(V).
\] 
\end{enumerate}
We call this extension \emph{faithful} if $\supp(M)=\emptyset$ implies $M\cong 0$, for arbitrary $M$ in $\Stab(\msc{C})$.
\end{definition}

Obviously faithfulness of a given extension requires that the original theory was faithful as well.

In the statement of Definition \ref{def:text} we have denoted the extension to $\Stab(\msc{C})$ simply by $(Y,\supp)$, by abuse of notation.  In practice, there are usually many possible choices of extensions for a given theory.  These various extensions may agree, or may not agree.  For example, for restricted representations $\msc{C}=\rep^{\rm res}(\mfk{n})$ of a nilpotent restricted Lie algebra in finite characteristic, one may consider the $\pi$-point extension \cite{friedlanderpevtsova07}, local cohomology extension \cite{bensoniyengarkrause08}, or hypersurface extension (see Section \ref{sect:tensor}).  All of these extensions are tensor extensions, and the $\pi$-point and local cohomology extensions can be shown to agree \cite{bensoniyengarkrausepevtsova18}, but neither of these extensions are known to agree with the hypersurface extension at the moment.

\begin{definition}\label{def:lavish}
A support theory $(Y,\supp)$ for $\stab(\msc{C})$ is called \emph{lavish} if it is exhaustive, multiplicative, and admits a faithful tensor extension to $\Stab(\msc{C})$.
\end{definition}

The production of--what we've termed--lavish support theories is, at the moment, the primary method employed in the classification of thick ideals for a given tensor triangulated category.  Of course, we have made some alterations to the standard framework to account for our generally non-braided settings.

\begin{remark}
The sidedness of (a) and (b) mirrors, non-coincidentally, the sidedness which appears in the definition of an integration of an integrable Hopf algebra, Definition \ref{def:integrable}.  So, in practice, the sidedness of (b1) and (b2) can be fixed across all $V$ and $M$, and the finite-dimensional object $V$ can be taken to act uniformly on the left (resp.\ on the right) in these formulas.  However, as far as proofs are concerned, it suffices to assume that for a given $V$, and a given $M$, at least one of the containments of (a) holds, and when $V$ centralizes the simples at least one of the equalities of (b) holds.
\end{remark}

\begin{remark}
One might compare Definition \ref{def:text} with that of a noncommutative support datam from \cite[Definition 4.1.1]{nakanovashawyakimov}.  The tensor relation employed in \cite{nakanovashawyakimov} can be seen, fundamentally, as a relation on products of thick ideals (rather than objects) \cite[Lemma 4.3.1]{nakanovashawyakimov}.
\end{remark}

\subsection{Localization functors}

We recall some information from~\cite{neeman01,bensoniyengarkrause08}.  An exact endomorphism $L:\mcl{T}\to \mcl{T}$ of a triangulated category $\mcl{T}$ is said to be a localization functor if there exists a natural transformation $\eta:id_\mcl{T}\to L$ so that $L\eta:L\to L^2$ is an isomorphism, and $L\eta=\eta L$.  We suppose also that, if $L$ vanishes on a collection $\{X_\alpha\}_\alpha$ in $\mcl{T}$, and the coproduct $\oplus_{\alpha} X_\alpha$ exists in $\mcl{T}$, then $L$ also vanishes on $\oplus_\alpha X_\alpha$.  In this case $L$ is paired with another functor $\Gamma:\mcl{T}\to \mcl{T}$ and transformation $\Gamma\to id_{\mcl{T}}$ so that, at any $X$ in $\mcl{T}$, we have a triangle
\begin{equation}\label{eq:1006}
\Gamma X\to X\to LX.
\end{equation}
Consider $\ker(L)$ the kernel of $L$ and $\ker(L)^\perp$ the collection of all objects which admit no non-zero maps from $\ker(L)$.  By~\cite[Lemma 3.3]{bensoniyengarkrause08} the functor $\Gamma$ has image in the subcategory $\ker(L)$, and $L$ has image in $\ker(L)^\perp$.
\par

Consider now $\mcl{T}$ a compactly generated triangulated category, and let $\msc{I}\subset \mcl{T}^c$ be a thick subcategory.  We say $\msc{I}$ admits a localization functor if there exists a localization $L_\msc{I}$ as above with $\mcl{T}^c\cap \mrm{ker}(L_\msc{I})=\msc{I}$.

\begin{lemma}[{\cite[Corollary 7.2.2]{krause10}}]\label{lem:L_K}
Let $\mcl{T}$ be a compactly generated triangulated category which admits arbitrary sums, and $\msc{I}\subset \mcl{T}^c$ be a thick subcategory of compact objects.  Then the localizing subcategory $\msc{I}_{\rm loc}$ generated by $\msc{I}$ in $\mcl{T}$ is such that $\msc{I}_{\rm loc}\cap \mcl{T}^c=\msc{I}$.
\end{lemma}

Brown representability now implies

\begin{theorem}[{\cite{rickard97},~\cite[Proposition 5.2.1]{krause10}}]\label{thm:L_K}
For $\mcl{T}$ a compactly generated triangulated category, any thick subcategory $\msc{I}$ in $\mcl{T}^c$ admits a localization functor $L_\msc{I}$ with $\ker(L_\msc{I})=\msc{I}_{\rm loc}$.  In particular, any thick ideal in $\mcl{T}^c$ admits a localization functor.
\end{theorem}

We recall that $\Stab(\msc{C})$ admits arbitrary sums, and has compact objects $\Stab(\msc{C})^c=\stab(\msc{C})$, so that the above general results are applicable to the situations under consideration here.

\subsection{Remarks on abelian, infinity, and triangulated categories}

Let us make some remarks on our setting.  In this study we operate in-between the (abelian) tensor category $\msc{C}$ and its triangulated counterpart $\stab(\msc{C})$.  So, we do certain operations at the abelian level, then observe how these operations manifest at the triangulated level.
\par

As far as we can tell, one cannot simply work at the triangulated level to reconstruct the analysis we present here.  This is not so shocking as it is known that triangulated categories are too coarse to admit certain refined constructions.  We would suggest instead that one should work with stable tensor $\infty$-categories, where constructions such as Drinfeld centralizers should be well-behaved, and also agree with their abelian counterparts (cf.\ \cite{benzvifrancisnadler10}).
\par

As another point for consideration, the $\infty$-category perspective should allow one to work directly with the bounded derived (rather than stable) category for $\msc{C}$, which one can consider as the subcategory of compacts in its $\Ind$-completion \cite[\S 5.3]{lurie09}.  With such a completion operation one could hope, in the long run, to produce a means of ``tautologically extending" supports from rigid tensor $\infty$-categories to their associated presentable ($\Ind$-)categories.  Such investigations seem impossible to pursue at the triangulated level.

\section{Classifying thick ideals, in the abstract}
\label{sect:abstract}

We consider the stable category of a finite tensor category $\msc{C}$, and discuss relations between extensions of support theories and the classification of thick ideals in $\stab(\msc{C})$.  We pay special attention to the case of a tensor category $\msc{C}$ with the so-called Chevalley property, which is the appropriate categorical framing for the examples of Section \ref{sect:examples}.  The main point of the section is to prove the following.

\begin{theorem}\label{thm:classify}
Suppose that $\msc{C}$ is a tensor category with the Chevalley property.  Then any lavish support theory for $\stab(\msc{C})$ classifies thick ideals (see Definition \ref{def:lavish}).
\end{theorem}

The proof of Theorem \ref{thm:classify} is given at the conclusion of Subsection \ref{sect:idunnowhattocallthis} below.

\subsection{Classifying thick ideals}
\label{sect:classify}

Recall from Section \ref{sect:stabC} the notions of central and centralizing objects in $\stab(\msc{C})$.  Below we prove the following

\begin{proposition}\label{prop:classify}
Consider a finite tensor category $\msc{C}$, and suppose that all thick ideals in $\stab(\msc{C})$ are generated by objects which centralize the simples.  Then any lavish support theory for $\stab(\msc{C})$ classifies thick ideals.
\end{proposition}

We collect a few supporting lemmas before proving the proposition.

\begin{lemma}
Consider any thick ideal $\msc{I}$ is $\stab(\msc{C})$.  An object $V$ in $\stab(\msc{C})$ lies in $\msc{I}$ if and only if its duals $V^\ast$ and ${^\ast V}$ lie in $\msc{I}$.
\end{lemma}

\begin{proof}
By the definition of a dual for an object in $\msc{C}$ \cite[\S 2.10]{egno15}, $V^\ast$ and ${^\ast V}$ appear as summands of $V^\ast\ot V\ot V^\ast$ and ${^\ast V}\ot V\ot {^\ast V}$ respectively.  Similarly, $V$ appears as a summand in $V\ot V^\ast\ot V$ and $V\ot {^\ast V}\ot V$.  Since $\msc{I}$ is closed under taking summands and the actions of $\stab(\msc{C})$, the claim follows.
\end{proof}

\begin{lemma}\label{lem:9000}
Consider a support theory $(Y,\supp)$ for $\stab(\msc{C})$ which is both exhaustive and multiplicative.  If all ideals in $\stab(\msc{C})$ are generated by central objects (resp.\ objects which centralize the simples), then any closed subset $\Theta$ in $Y$ is realizable as the support $\Theta=\supp(V)$ of a central object $V$ (resp.\ an object $V$ which centralizes the simples).
\end{lemma}

\begin{proof}
Since any closed subset in $Y$ decomposes as a union of irreducible closed subsets, it suffices to show that any irreducible closed subset in $Y$ is realizable as the support of such a central object.  Take any irreducible closed subset $\Theta$ in $Y$, and $V$ with $\Theta=\supp(V)$.  Consider the thick ideal $\langle V\rangle^{\ot}$ and a sufficiently central generating set $\{V_i\}_{i\in I}$.  (Sufficiently central meaning either central, or centralizing the simples, respectively.)  Then
\[
\cup_i\supp(V_i)=\supp(\langle V\rangle^{\ot})=\supp(V).
\]
It follows that the generic point in $\Theta$ is in $\supp(V_i)$, for some $i$, and subsequently that $\Theta=\supp(V_i)$.
\end{proof}

We now offer the proof of Proposition \ref{prop:classify}.

\begin{proof}[Proof of Proposition \ref{prop:classify}]
We first prove that an inclusion of supports $\supp(V)\subset \supp(W)$, for $V$ and $W$ centralizing the simples, implies an inclusion $\langle V\rangle^{\ot}\subset \langle W\rangle^{\ot}$.  Consider $L_W:\Stab(\msc{C})\to \Stab(\msc{C})$ the localization functor associated to the localizing ideal $\langle W\rangle^\ot_{\rm loc}$ generated by $W$.  So, we have $L_W(V)=0$ if and only if $V$ is in $\langle W\rangle^\ot_{\rm loc}$, and hence if and only if $V$ is in the thick ideal generated by $W$, by Lemma \ref{lem:L_K}.
\par

We consider the given extension of $(Y,\supp)$ to all of $\Stab(\msc{C})$.  Via the triangle $L_W(V)\to V\to \Gamma_W(V)$, and the fact that $\Gamma_W(V)$ is in the localizing ideal generated by $W$, we see that $\supp(L_W(V))\subset \supp(W)$. We suppose, arbitrarily, that the relation of Definition \ref{def:text} (b2) hold when tensoring by a finite-dimensional representation on the \emph{left}.  The argument when this relations hold on the right is completely similar.
\par

Recall that $\Lambda$ is the sum of the simples for $\msc{C}$, with multiplicity.  We have
\[
\Hom_{\Stab(\msc{C})}(\Sigma^{-n}\Lambda,W\ot L_W(V))=\Hom_{\Stab(\msc{C})}(\Sigma^{-n}W^\ast\ot \Lambda,L_W(V))=0,
\]
so that $W\ot L_W(V)=0$, since $\Lambda$ generates the big stable category.  But now
\[
\emptyset=\supp(W\ot L_W(V))=\supp(W)\cap\supp(L_W(V))=\supp(L_W(V)),
\]
since $L_W(V)$ lies in the localizing ideal generated by $W$, and hence has support contained in $\supp(W)$.  It follows that $L_W(V)=0$ and therefore that $V\subset \langle W\rangle^\ot$, as desired.  So we now have the implication
\begin{equation}\label{eq:776}
\supp(V)\subset \supp(W)\ \Rightarrow\ \langle V\rangle^{\ot}\subset \langle W\rangle^{\ot},
\end{equation}
for $V$ and $W$ which centralize the simples.
\par

Now, our central generation hypothesis and the implication \eqref{eq:776} gives
\[
\msc{I}_{\supp(\msc{I})}=\msc{I}.
\]
Furthermore, since any closed subset in $Y$ is realized as the support of an object in $\stab(\msc{C})$, we have $\supp(\msc{I}_\Theta)=\Theta$.  So we have established the desired bijection between thick ideals and specialization closed subsets in $Y$.
\par

In the event that $\Theta$ is closed, we have $\Theta=\supp(V)$ for some $V$ which centralizes the simples, by Lemma \ref{lem:9000}.  Hence the inclusion \eqref{eq:776} implies $\msc{I}_\Theta=\langle V\rangle^{\ot}$, so that $\msc{I}_\Theta$ is finitely generated.  Conversely, if $\msc{I}$ is generated by some finite collection of objects $\{V_i\}_{i=1}^n$ then $\supp(\msc{I})=\cup_{i=1}^n\supp(V_i)$ is closed.  Thus the aforementioned bijection restricts to a bijection between compactly generated ideals in $\stab(\msc{C})$ and closed subsets in $Y$.
\end{proof}

\begin{remark}
Our specific use of localization functors is informed by the arguments of \cite[Proof of Theorem 7.4.1]{boekujawanakano}, which in turn follow arguments introduced by Rickard \cite{rickard97}.
\end{remark}

\subsection{Classifying thick ideals in Chevalley categories}
\label{sect:idunnowhattocallthis}

Suppose now that $\msc{C}$ is Chevalley, aka has the Chevalley property.  By this we mean that the subcategory of semisimple objects in $\msc{C}$ is a \emph{tensor} subcategory.  Let $\msc{D}$ denote the tensor subcategory of semisimple objects in $\msc{C}$.

The Chevalley condition on $\msc{C}$ can be seen as a kind of solvability condition (although the notion of a solvable tensor category is already taken \cite[Definition 9.8.1]{egno15}), and the representation categories of all of the examples under consideration here, (\hyperref[it:F1]{F1})--(\hyperref[it:F3]{F3}) \& (\hyperref[it:G]{G}), are Chevalley.  For Chevalley categories, the centralization hypotheses of Proposition \ref{prop:classify} are always satisfied, as we now show.

\begin{lemma}\label{lem:1071}
Suppose $\msc{C}$ is Chevalley.  Then any thick ideal $\msc{I}$ in $\stab(\msc{C})$ is generated by objects which centralize the simples.
\end{lemma}

\begin{proof}
As above, let $\msc{D}\subset \msc{C}$ denote the tensor subcategory of semisimple objects.  Let $\{\lambda_i\}_{i=1}^n$ be a complete list of simples in $\msc{C}$, and take $\lambda^i=(\lambda_i)^\ast$.  It suffices to show that for each object $V$ in $\stab(\msc{C})$, there is some object $W$ which centralizes the simples, lies in the thick ideal generated by $V$, and is such that $V$ appears as a summand in $W$.
\par

In the Hopf case, where $\msc{C}=\rep(\msf{u})$, $\msc{D}=\rep(\Lambda)$, and $Z^\msc{D}(\msc{C})=\rep(\msf{u}\bowtie \Lambda^\ast)$, one can deduce explicitly that the right adjoint $R:\msc{C}\to Z^\msc{D}(\msc{C})$ is of the form
\[
R(V)=\Lambda\bowtie V=\sum_i\lambda_i\ot V\ot \lambda^i,
\]
via the bimodule decomposition $\Lambda=\oplus_i \End_k(\lambda_i)=\sum_i\lambda_i\ot(\lambda_i)^\ast$.  So we see that $V$ is a summand of $FR(V)$, where $F$ is the forgetful functor, and that $FR(V)$ is a object in the thick ideal generated by $V$ which centralizes the simples.
\par

We claim that the formula, $FR(V)=\sum_i\lambda_i\ot V\ot \lambda^i$, holds in general.  Or, more to the point, we claim that for any $V$ in a Chevalley tensor category $\msc{C}$ the object $\sum_i \lambda_i\ot V\ot \lambda^i$ centralizes the simples.  Indeed, this follows by \cite[Lemma 4.5]{shimizu19II}.  We sketch the details, following \cite[Section 3.2]{shimizu19} (see also \cite{bruguieresvirelizier12}).
\par

Consider $o:\msc{C}\to \msc{C}$ the functor $o(V)=\sum_{i=1}^n \lambda_i\ot V\ot\lambda^i$.  At each $V$ the object $o(V)$ is naturally identified with the end $\int_{X\in \msc{D}}X\ot V\ot X^\ast$ of the functor $-\ot V\ot (-)^\ast:\msc{D}\ot\msc{D}^{op}\to \msc{C}$ \cite[Remark 3.2]{bruguieresvirelizier13}.  This identification provides each $o(V)$ with natural maps $\pi_{V,X}:o(V)\to X\ot V\ot X^\ast$ for all semisimple $X$, and provides the functor $o$ with a comonad structure $\delta:o\to o^2$.  We then obtain a centralizing structure $\gamma_V:o(V)\ot -\to -\ot o(V)$ via the composite
\[
\begin{array}{l}
\gamma_{V,X}:=(id_{X\ot o(V)}\ot ev_{\lambda})\circ (\pi_{o(V),X}\ot id_X)(\delta_V\ot id_X):\vspace{1mm}\\
\hspace{2cm}o(V)\ot X\to o^2(V)\ot X\to X\ot o(V)\ot X^\ast\ot X\to X\ot o(V).
\end{array}
\] 
\end{proof}

We now prove the main result of the section.

\begin{proof}[Proof of Theorem \ref{thm:classify}]
Lemma \ref{lem:1071} obviates the hypotheses of Proposition \ref{prop:classify}, so that the result follows directly from Proposition \ref{prop:classify}.
\end{proof}

\section{Thick ideals and spectra for finite tensor categories}
\label{sect:spec}

We explain how a support theory $(Y,\supp)$ which classifies thick ideals in $\stab(\msc{C})$ also provides, under ideal circumstances, a calculation of the spectrum of prime ideals for $\stab(\msc{C})$.  As always, $\msc{C}$ denotes a finite tensor category, and we pay special attention to the Chevalley case.  We note that all of the materials of this section are adapted, almost directly, from work of Balmer \cite{balmer05}.

\subsection{Prime ideals}

\begin{definition}
A thick ideal $\msc{P}\subset \stab(\msc{C})$ is called a prime ideal if $\msc{P}$ is a proper ideal in $\stab(\msc{C})$ and, for any thick ideals $\msc{I}$ and $\msc{J}$ in $\stab(\msc{C})$ with $\msc{I}\ot\msc{J}\subset \msc{P}$, either $\msc{I}\subset \msc{P}$ or $\msc{J}\subset \msc{P}$.
\end{definition}

This notion of a prime ideal in $\stab(\msc{C})$ is the na\"ive generalization of the notion of a prime ideal in a noncommutative ring, and is also employed in \cite{nakanovashawyakimov}.  Our notion of the spectrum of a finite tensor category, given below, also agrees with that of \cite{nakanovashawyakimov}.

\subsection{The spectrum of prime ideals}

\begin{definition}\label{def:prime}
The spectrum of $\stab(\msc{C})$ is the collection of prime ideals in $\stab(\msc{C})$,
\[
\cSpec(\msc{C})\big(=\cSpec(\stab(\msc{C}))\big):=\{\msc{P}\subset \stab(\msc{C}):\msc{P}\text{ is a thick prime ideal}\}.
\]
\end{definition}

The collection $\cSpec(\msc{C})$ has the natural structure of a ringed space--although we generally deal only with its topology.  The topology on $\cSpec(\msc{C})$ is defined as follows: A basic closed subset in $\cSpec(\msc{C})$ is a set of the form
\[
\supp^{uni}(V):=\{\msc{P}\in \cSpec(\msc{C}):V\notin \msc{P}\},
\]
where $V$ is an object in $\stab(\msc{C})$.  The basic opens are then $U_V:=\cSpec(\msc{C})-\supp^{uni}(V)$.  So $U_V$ is the collection of primes which contain $V$.  We have $U_V\cap U_W=U_{V\oplus W}$, $U_0=\cSpec(\msc{C})$, and $U_1=\emptyset$, so that these subsets form a basis of a uniquely determined topology on $\cSpec(\msc{C})$.
\par

For any open set $U\subset \cSpec(\msc{C})$ we consider the thick ideal
\[
\msc{K}_U:=\{V:\supp^{uni}(V)\cap U=\emptyset\},
\]
and corresponding quotient category $\stab(\msc{C})/\msc{K}_U$.  The sheaf of rings $\O_\msc{C}$ on $\cSpec(\msc{C})$ is then the sheafification of the presheaf
\[
\O_\msc{C}^{\rm pre}:U\mapsto \Hom_{\stab(\msc{C})/\msc{K}_U}(\1,\1).
\]

\begin{proposition}[{cf.\ \cite{balmer05}}]\label{prop:properties}
Suppose that $\msc{C}$ is a non-semisimple, finite, tensor category, and that all ideals in $\stab(\msc{C})$ are centrally generated.  Then 
\begin{enumerate}
\item $\cSpec(\msc{C})$ is a non-empty $\mrm{T}_0$ topological space.
\item $\cSpec(\msc{C})$ is quasi-compact and contains a closed point.
\item All basic opens $U_V$ in $\cSpec(\msc{C})$ are quasi-compact.
\item Irreducible closed sets in $\cSpec(\msc{C})$ have a unique generic point.
\item Pulling back along the functor $F:\stab(Z(\msc{C}))\to \stab(\msc{C})$ induces an injective map of ringed spaces $F^\ast:\cSpec(\msc{C})\to \cSpec(Z(\msc{C}))$.
\end{enumerate}
\end{proposition}

As we do not use the above results, we do not give a formal proof.  However, points (1)--(4) are proved just as in \cite{balmer05}, where one employs central generation to restrict all computations to computations with central objects.  Statement (5) follows from the fact that any two centrally generated ideals $\msc{P}$ and $\msc{P}'$ in $\stab(\msc{C})$ agree if and only if their preimages in $\stab(Z(\msc{C}))$ agree, via central generation.
\par

\begin{remark}
It seems likely at this point that $\cSpec(\msc{C})$ should in fact be identified with a \emph{closed subspace} in $\cSpec(Z(\msc{C}))$.  Indeed, it should be the support of the induction $R(\1)$ of the identity $\1$ in $\msc{C}$, where $R:\msc{C}\to Z(\msc{C})$ is the right adjoint to the forgetful functor.  Let $\Theta_\msc{C}\subset \cSpec(Z(\msc{C}))$ denote the support of $R(\1)$ in $\cSpec(Z(\msc{C}))$.  We would suggest furthermore that the structure sheaf $\O_\msc{C}$ on $\cSpec(\msc{C})\cong_{homeo} \Theta_\msc{C}$ is a finite extension of the sheaf of rings $\O_{Z(\msc{C})}$ in general, so that $(\cSpec(\msc{C}),\O_\msc{C})$ is not locally ringed, but semi-locally ringed in general.  (See Section \ref{sect:not_lr}.)
\end{remark}

\begin{remark}
Relative to the perspective of \cite{buankrausesnashallsolberg20}, we are suggesting that a central action $R\to \End(id_{\stab\msc{C}})$ should apparently factor through the action of the Drinfeld center $\Ext^\ast_{Z(\msc{C})}(\1,\1)\to \End(id_{\stab\msc{C}})$ if one hopes to classify two-sided thick ideals.
\par

One can also see \cite{vashaw} for an exploration of the relationship between spectra for $\stab(\msc{C})$ and spectra for $\stab(Z(\msc{C}))$ in some extreme settings.
\end{remark}

\subsection{The spectrum and support}
\label{sect:univ}

\begin{lemma}\label{lem:1043}
Suppose that $V$ and $W$ are in $\stab(\msc{C})$, and that one of $V$ or $W$ centralizes the simples.  Then in $\cSpec(\msc{C})$ we have
\[
\supp^{uni}(V\ot W)=\supp^{uni}(V)\cap\supp^{uni}(W).
\]
More generally, for arbitrary $V$ and $W$, we have a containment
\[
\supp^{uni}(V\ot W)\subset (\supp^{uni}(V)\cap\supp^{uni}(W)).
\]
That is to say, the pair $(\cSpec(\msc{C}),\supp^{uni})$ is a multiplicative support theory for $\stab(\msc{C})$.
\end{lemma}

\begin{proof}
For the inclusion $\supp^{uni}(V\ot W)\subset (\supp^{uni}(V)\cap\supp^{uni}(W))$ we simply note that if $V\ot W$ is not in a prime $\msc{P}$ then neither $V$ nor $W$ can be in $\msc{P}$.
\par

Now, in general we have
\begin{equation}\label{eq:773}
\langle \langle V\rangle^{\ot}\ \ot\langle W\rangle^{\ot}\rangle^{\ot}=\langle \oplus_i (V\ot\lambda_i\ot W)\rangle^{\ot},
\end{equation}
where the sum runs over a complete list of the simples in $\msc{C}$.  To see this, note that we have the exact endomorphism $V\ot-:\stab(\msc{C})\to \stab(\msc{C})$ of right $\msc{C}$-module categories, so that the preimage of the thick right ideal generated by $V\ot(\oplus_i \lambda_i\ot W)$ contains the thick right ideal generated by $\oplus_i\lambda_i\ot W$.  But the thick right ideal generated by this sum is the two sided ideal generated by $W$.  Hence 
\begin{equation}\label{eq:1522}
V\ot \langle W\rangle^{\ot}\subset \langle \oplus_i (V\ot\lambda_i\ot W)\rangle^{\ot}.
\end{equation}
Similarly, for any $W'\in \langle W\rangle^{\ot}$ we have
\begin{equation}\label{eq:1526}
\langle V\rangle^{\ot}\ot W'\subset \langle \oplus_i (V\ot\lambda_i\ot W')\rangle^{\ot}\subset \langle \oplus_i (V\ot\lambda_i\ot W)\rangle^{\ot},
\end{equation}
where the final inclusion follows from \eqref{eq:1522} and the fact that $\oplus_i\lambda_i\ot W'\in \langle W\rangle^{\ot}$.  Hence we have an inclusion
\[
\langle \langle V\rangle^{\ot}\ \ot\langle W\rangle^{\ot}\rangle^{\ot}\subset \langle \oplus_i (V\ot\lambda_i\ot W)\rangle^{\ot}.
\]
The opposite inclusion is clear, as $\oplus_i V\ot\lambda_i\ot W$ is in $\langle V\rangle^{\ot}\ot \langle W\rangle^{\ot}$.  This verifies the formula \eqref{eq:773}.
\par

In the case that $V$ centralizes the simples, we have $\oplus_i V\ot\lambda_i\ot W\cong\oplus_i \lambda_i\ot V\ot W$ so that $\langle (\oplus_i V\ot\lambda_i\ot W)\rangle^{\ot}=\langle V\ot W\rangle^{\ot}$.  One has a similar equality when $W$ centralizes the simples.  So, for any prime $\msc{P}$, we have that
\[
\begin{array}{rll}
V\ot W\in \msc{P}& \Leftrightarrow\ \langle V\rangle^{\ot}\ot \langle W\rangle^{\ot}\subset \msc{P} & \text{(by our centralizing hypotheses)}\\
& \Rightarrow\  V\in\msc{P}\text{ or }W\in \msc{P} &\text{(via primeness of }\msc{P}).
\end{array}
\]
So we deduce the claimed equality of supports.
\end{proof}

We record an observation from the proof, which is also implicit in the equality $\supp^{uni}(V\ot W)=\supp^{uni}(V)\cap\supp^{uni}(W)$ at $V$ or $W$ centralizing the simples.

\begin{lemma}\label{lem:859}
Consider $V$ and $W$ in $\stab(\msc{C})$, with one of $V$ or $W$ centralizing the simples, and let $\msc{P}$ be a prime in $\stab(\msc{C})$.  Then the product $V\ot W$ lies in $\msc{P}$ if and only if $V$ lies in $\msc{P}$ or $W$ lies in $\msc{P}$.
\end{lemma}

\begin{remark}
Although one can easily see this point directly, Lemma \ref{lem:859} tells us that thick prime ideals in the stable category $\stab(\msc{C})$ for \emph{braided} $\msc{C}$, defined as in Definition \ref{def:prime}, are primes in the sense of \cite{balmer05}.  So our prime ideal spectrum agrees with Balmer's spectrum \cite[Definition 2.1]{balmer05} in this case.
\end{remark}

The following lemma says that the support theory $(\cSpec(\msc{C}),\supp^{uni})$ is universal--in particular terminal--among multiplicative support theories for $\stab(\msc{C})$, provided ideals in $\stab(\msc{C})$ are all generated by objects which centralize the simples.

\begin{lemma}[{cf.\ \cite[Theorem 3.2]{balmer05}}]\label{lem:776}
Suppose that $(Y,\supp)$ is a multiplicative support theory for $\stab(\msc{C})$.  Suppose additionally that all ideals in $\stab(\msc{C})$ are generated by objects which centralize the simples.  Then there is a corresponding continuous map
\[
f_{\supp}:Y\to \cSpec(\msc{C}),\ \ y\mapsto \msc{P}_y:=\{W:y\notin \supp(W)\}.
\]
This map satisfies $f^{-1}_{\supp}(\supp^{uni}(V))=\supp(V)$, for all $V$ in $\stab(\msc{C})$.
\end{lemma}

\begin{proof}
We first prove that $\msc{P}_y$ is in fact prime.  Consider two ideals $\msc{I}$ and $\msc{J}$ in $\stab(\msc{C})$, and consider generators $\{V_i\}_i$ and $\{W_j\}_j$ for $\msc{I}$ and $\msc{J}$ respectively which centralize the simples.  Suppose that $\msc{I}\ot \msc{J}\subset \msc{P}_y$ and that $V_i\notin \msc{P}_y$ for some $i$.  Since $V_i\ot W_j\in \msc{P}_y$ for all $j$ we have $y\notin \supp(V_i\ot W_j)$ while $y\in \supp(V_i)$.  The tensor product property $\supp(V_i\ot W_j)=\supp(V_i)\cap\supp(W_j)$ then forces $y\notin \supp(W_j)$ for all $j$.  It follows that $\msc{J}\subset \msc{P}_y$.  So we see that $\msc{P}_y$ is prime.
\par

For continuity, as well as the final claim $f^{-1}_{\supp}(\supp^{uni}(V))=\supp(V)$, we have directly
\[
\supp^{uni}(V)\cap f_{\rm supp}(Y)=\{\msc{P}_y:V\notin\msc{P}_y\}=\{\msc{P}_y:y\in \supp(V)\},
\]
so that $\supp(V)=f^{-1}(\supp^{uni}(V))$.  One takes complements to find that the preimage of a basic open $f^{-1}_{\supp}(U_V)$ is the open complement $Y-\supp(V)$, and hence that $f_{\supp}$ is continuous.
\end{proof}

In the Chevalley case the hypotheses of Lemma \ref{lem:776} are met immediately, so that we have

\begin{corollary}\label{cor:chev_776}
Suppose that $\msc{C}$ is Chevalley.  Then for any multiplicative support theory $(Y,\supp)$ for $\stab(\msc{C})$, there is an continuous map
\[
f_{\supp}:Y\to \cSpec(\msc{C}),\ \ y\mapsto \msc{P}_y:=\{W:y\notin \supp(W)\}
\]
which satisfies $f^{-1}_{\supp}(\supp^{uni}(V))=\supp(V)$, for all $V$ in $\stab(\msc{C})$.
\end{corollary}

\begin{proof}
In this case all ideals in $\stab(\msc{C})$ are generated by objects which centralize the simples, by Lemma \ref{lem:1071}.  So the claim follows by Lemmas \ref{lem:776}.
\end{proof}

\subsection{Calculating the spectrum}
\label{sect:balmer}

\begin{theorem}[{cf.\ \cite[Theorem 5.2]{balmer05}}]\label{thm:balmer}
Consider $\msc{C}$ a finite tensor category, and suppose that all thick ideals in $\stab(\msc{C})$ are generated by objects which centralize the simples.  If $(Y,\supp)$ is an lavish support theory, then the map $f_{\supp}$ of Lemma \ref{lem:776} is a homeomorphism,
\[
f_{\supp}:Y\overset{\cong}\longrightarrow \cSpec(\msc{C}).
\]
\end{theorem}

By Proposition \ref{prop:classify}, any lavish support theory is, in the language of \cite[Definition 5.1]{balmer05}, a classifying support data.  This classifying property is the technically relevant point.  For the proof of the theorem, one literally copies the proof of \cite[Theorem 5.2]{balmer05}, with only a slight modification for the surjectivity argument.  For the sake of completeness, we repeat Balmer's proof here.

\begin{proof}
We have already assumed that arbitrary closed subsets in $Y$ are realizable as supports of objects, and, by Lemma \ref{lem:9000}, central generation implies that all closed subsets are realizable as supports of objects which centralize the simples.
\par

For injectivity, consider for $y\in Y$ the specialization closed subset $\Theta(y):=\{x\in Y:y\notin \overline{x}\}$ in $Y$.  Note that this subset $\Theta(y)$ does not contain $y$, so that if an object $V$ in $\stab(\msc{C})$ has support contained in $\Theta(y)$ then $y\notin \supp(V)$.  Conversely, since $\supp(V)$ is closed in $Y$, and therefore contains the closures of all its points, we see that $y\notin \supp(V)$ then $\supp(V)\subset \Theta(y)$.  So we calculate the thick ideal defined by $\Theta(y)$ as
\[
\msc{I}_{\Theta(y)}=\{V:y\notin \supp(V)\}=\msc{P}_y=f_{\supp}(y).
\]
Now, since $(Y,\supp)$ admits a tensor extension to the big stable category, it also classifies thick ideals in $\stab(\msc{C})$ by Proposition \ref{prop:classify}.  Therefore
\[
f_{\supp}(x)=f_{\supp}(y) \ \Leftrightarrow \ \Theta(x)=\Theta(y)\ \Rightarrow\ x\in \overline{y}\ \text{and}\ y\in\overline{x}\ \Rightarrow\  \overline{x}=\overline{y}.
\]
Since $Y$ is $T_0$, this implies $x=y$.  So $f_{\supp}$ is injective.
\par

For surjectivity, consider a prime $\msc{P}$ in $\stab(\msc{C})$ and its associated specialization closed subset $\Theta=\supp(\msc{P})$ in $Y$.  Note that, since $\msc{P}$ is a proper ideal in $\stab(\msc{C})$, $\Theta$ is a proper subset in $Y$.  We claim that the associated collection of closed subsets in $Y$,
\[
\mcl{F}=\{\overline{x}:x\notin \Theta\},
\]
is nonempty and admits a minimal element (under inclusion).
\par

Suppose that such $x_0$ with minimal closure $\overline{x_0}\in \mcl{F}$ exists.  Then $x_0$ provides an identification $Y-\Theta=\{x:x_0\in \overline{x}\}$.  Indeed, minimality implies immediately the containment $Y-\Theta\subset \{x:x_0\in \overline{x}\}$ and the opposite containment holds as $x$ itself is not in $\Theta$, and $\Theta$ is specialization closed.  But this set $\{x:x_0\in \overline{x}\}$ is just $\Theta(x_0)$ from above.  Therefore
\[
\msc{P}=\msc{I}_\Theta=\msc{I}_{\Theta(x_0)}=f(x_0),
\]
and we observe surjectivity of our map.
\par

So, in order to prove surjectivity, it suffices to show that $\mcl{F}$ does in fact admit a minimal element.  Consider points in the complement $x,y\in Y-\Theta$, and objects $V$ and $W$ which centralize the simples and have supports $\supp(V)=\overline{x}$, $\supp(W)=\overline{y}$.  Since these objects are not supported in $\Theta$, we have that neither $V$ nor $W$ lies in $\msc{P}$, and hence their product is not in $\msc{P}$.  It follows that
\[
\overline{x}\cap\overline{y}=\supp(V\ot W)\nsubseteq \Theta.
\]
So there exists an element $z$ in this intersection, which also lies in the complement $Y-\Theta$.  Rather, any two subsets $\overline{x}$ and $\overline{y}$ in $\mcl{F}$ admit a third subset $\overline{z}$ in $\mcl{F}$ with $\overline{z}$ contained in both $\overline{x}$ and $\overline{y}$.  Since $Y$ is Noetherian, this implies that $\mcl{F}$ admits a minimal element, as desired.  We therefore have surjectivity of $f_{\supp}$, and conclude that $f_{\supp}$ is bijective.
\par

Finally, to see that $f_{\supp}$ is a homeomorphism, note that the basic closed sets $\supp^{uni}(V)$ in $\cSpec(\msc{C})$ are identified with the closed sets $\supp(V)$ in $Y$ under $f_{\supp}$.  Since the $\supp(V)$ exhaust all closed subsets in $Y$, we see that the topology on $Y$ is also generated by the $\supp(V)$, and hence the topologies on $Y$ and $\cSpec(\msc{C})$ agree.
\end{proof}

One now applies Lemma \ref{lem:1071} in the Chevalley case to find

\begin{theorem}\label{thm:chevalley_balmer}
Consider $\msc{C}$ a finite, Chevalley tensor category.  Then for any lavish support theory $(Y,\supp)$, the map $f_{\supp}$ of Corollary \ref{cor:chev_776} is a homeomorphism,
\[
f_{\supp}:Y\overset{\cong}\longrightarrow \cSpec(\msc{C}).
\]
\end{theorem}

\section{Hypersurface support as a tensor extension for (\hyperref[it:F1]{F1})--(\hyperref[it:F3]{F3})}
\label{sect:tensor}

We provide a detailed discussion of hypersurface support for Hopf algebras which are similar to our examples (\hyperref[it:F1]{F1})--(\hyperref[it:F4]{F4}).  Specifically, we discuss Hopf algebras with the Chevalley property, and local integrable Hopf algebras.
\par

The findings of the present section are employed in Section \ref{sect:ideals_q} in order to classify thick ideals in the stable categories of the families (\hyperref[it:F1]{F1})--(\hyperref[it:F3]{F3}).  We provide an independent discussion of categories of sheaves on non-connected group schemes in Sections \ref{sect:hyper_G} and \ref{sect:ideals_G}.

\subsection{The thick subcategory Lemma for hypersurfaces}
\label{sect:tsl}

Consider integrable $\msf{u}$ with either conormal or local integration $\msf{U}\to \msf{u}$.  We say a hypersurface algebra $\msf{U}_K/(f)$ for $\msf{U}$ satisfies the \emph{thick subcategory lemma} if for each finitely generated module $N$ over $\msf{U}_K/(f)$ the following implication holds:
\[
\begin{array}{l}
N\text{ is non-perfect over }U_K/(f)\ \ ({\rm i.e.\ } \operatorname{projdim}_{U_K/(f)}(N)=\infty)\vspace{1mm}\\
\hspace{2cm}\Rightarrow\ K\in \langle \lambda\ot N\ot \mu:\lambda,\mu\in \operatorname{Irrep}(\msf{u}_K)\rangle\ \subset\  D_{coh}(\msf{U}_K/(f)).
\end{array}
\]
Here $f\in m_{Z_K}$ is any element with nonzero reduction $\bar{f}\in (m_Z/m_Z^2)_K$, and $D_{coh}(\msf{U}_K/(f))$ denotes the derived category of complexes with finitely generated cohomology.  Note that one needs to consider a conormal, or local, integration here so that the simples in $\rep(\msf{u})$ act both on the left and the right of the hypersurface category $\msf{U}_K/(f)$-mod.  The following was essentially proved in \cite{negronpevtsova}.

\begin{proposition}\label{prop:351}
For $\msf{u}$ of the types {\rm(\hyperref[it:F1]{F1})--(\hyperref[it:F4]{F4})}, all hypersurface algebras for the associated integration $\msf{U}\to \msf{u}$ satisfy the thick subcategory Lemma.
\end{proposition}

\begin{proof}
By initially changing base we may assume $K=k$.  The case $\msf{u}=\O(\mcl{G})$, for connected $\mcl{G}$, is already known \cite{takahashi10}.  So we deal with the cases (\hyperref[it:F1]{F1})--(\hyperref[it:F3]{F3}).
\par

We argue as in the proof of \cite[Lemma 13.9]{negronpevtsova}.  For each of the given examples we have $\msf{U}=\msf{U}^+\rtimes \Lambda$, where $\Lambda=\msf{u}/\operatorname{Jac}(\msf{u})$ and $\msf{U}^+$ is local.  The local algebra $\msf{U}^+$ is the completion of $A_q^+$, $U^{DK}_q(\mfk{n})$, and $\O(B)\rtimes U(\mfk{n})$ for the cases of the quantum complete intersection, quantum Borel, and double of $B_{(1)}$ respectively.  It was shown in \cite{negronpevtsova} that for finitely generated modules over hypersurface algebras
\begin{equation}\label{eq:327}
N\text{ non-perfect over }\msf{U}/(f)\ \ \Rightarrow\ \ k\in \langle \lambda\ot F(N):\lambda\in \operatorname{Irrep}(\Lambda)\rangle\subset D_{coh}(\msf{U}^+/(f)),
\end{equation}
where $F:D_{coh}(\msf{U}/(f))\to D_{coh}(\msf{U}^+/(f))$ is restriction functor.  More specifically, we showed that the augmentation ideal in $\msf{U}^+$ is generated by a $q$-regular sequence \cite[Definition 12.1]{negronpevtsova}, so that the implication follows by \cite[Lemmas 12.4, 12.5, \& 13.1]{negronpevtsova}.
\par

The right adjoint $R:\msf{U}^+/(f)\text{-mod}_{fg}\to \msf{U}/(f)\text{-mod}_{fg}$ is an exact map of left $\rep(\Lambda)$-module categories \cite[\S 3.3]{etingofostrik04}, and on modules $N$ over $\msf{U}/(f)$ we have
\[
R(F(N))\cong N\ot\Lambda^\ast\cong \oplus_{\mu\in \operatorname{Irrep}(\msf{u})}N\ot\mu.
\]
So this functor derives immediately to provide a map
\[
R:D_{coh}(\msf{U}^+/(f))\to D_{coh}(\msf{U}/(f))
\]
of triangulated module categories.  We therefore apply $R$ to the formula \eqref{eq:327}, and note that $k$ is a summand of $R(k)$, to find
\[
N\text{ non-perfect over }\msf{U}/(f)\ \Rightarrow\ k\in \langle \lambda\ot N\ot \mu:\lambda,\mu\in \operatorname{Irrep}(\Lambda)\rangle\ \subset D_{coh}(\msf{U}/(f)).
\]
\end{proof}

\subsection{Sub-multiplicativity of hypersurface support}

\begin{lemma}\label{lem:weak}
Consider an integrable Hopf algebra $\msf{u}$ with a given integration $\msf{U}\to \msf{u}$ which is either conormal or local.  Then for arbitrary $M$ and $N$ in $\Stab(\msf{u})$ there is an inclusion
\[
\supp^{hyp}_\mbb{P}(M\ot N)\subset \big(\supp^{hyp}_\mbb{P}(M)\cap\supp^{hyp}_\mbb{P}(N)\big).
\]
\end{lemma}

\begin{proof}
Consider $c$ any point in $\mbb{P}(m_Z/m_Z^2)$, which we may assume is closed by base change.  In the conormal case this just follows from the fact that we have an exact right adjoint $\Hom_k(N,-)$ to the action functors $N\ot-,\ -\ot N:D^b(\msf{U}_c)\to D^b(\msf{U}_c)$, at arbitrary $N$ in $\Rep(\msf{u})$.  (Rather, we have two adjoints corresponding to the actions on $\Hom_k(N,-)$ provided by the antipode and its inverse.)  Hence $N\ot-$ and $-\ot N$ preserve objects of finite projective dimension.  In the local case this follows from the fact that $M\ot N$ is in the localizing subcategory generated by $M$, and also in the localizing subcategory generated by $N$, in $\Sing(\msf{U}_c)$.
\end{proof}

We recall, from Section \ref{sect:hyp}, that a given integration $\msf{U}\to \msf{u}$ specifies a map $\kappa:\msf{Y}\to \mbb{P}(m_Z/m_Z^2)$ from the projective spectrum $\msf{Y}$ of the extensions $\Ext^\ast_\msf{u}(k,k)$.  Applying the above lemma in the case $V=\1$, and Proposition \ref{prop:538}, gives

\begin{corollary}\label{cor:contain}
Suppose that $\msf{u}$ admits an integration which is either conormal or local.  Then for any $M$ in $\Rep(\msf{u})$, the support $\supp^{hyp}_\mbb{P}(M)$ is contained in the closed subvariety $\kappa(\msf{Y})$ in $\mbb{P}(m_Z/m_Z^2)$.  That is to say, hypersurface support is a support theory valued in $\kappa(\msf{Y})$.
\end{corollary}

\begin{proof}
Proposition \ref{prop:538} says that $\supp^{hyp}_\mbb{P}(\1)=\kappa(\msf{Y})$, so that Lemma \ref{lem:weak} implies the claimed containment.
\end{proof}

In the case in which $\kappa:\msf{Y}\to\mbb{P}$ is a closed embedding, we therefore find that hypersurface support provides a \emph{triangular} extension of cohomological support from $\stab(\msf{u})$ to $\Stab(\msf{u})$.  One can show further that, in this case, all sub\emph{sets} of $\msf{Y}$ are realizable as the support of a module over $\msf{u}$.

\subsection{Hypersurface support as a tensor extension for (\hyperref[it:F1]{F1})--(\hyperref[it:F4]{F4})}
\label{sect:tpp}

\begin{theorem}\label{thm:tpp}
Consider integrable $\msf{u}$ with conormal integration $\msf{U}\to \msf{u}$.  Suppose that all hypersurfaces for $\msf{U}$ satisfy the thick subcategory lemma (see \ref{sect:tsl}).  Then hypersurface support for $\Stab(\msf{u})$ satisfies the tensor product property
\begin{equation}\label{eq:603}
\begin{array}{l}
\supp^{hyp}_\mbb{P}(V\ot M)=\supp^{hyp}_\mbb{P}(V)\cap\supp^{hyp}_\mbb{P}(M)\vspace{2mm}\\
\hspace{2cm}\text{and}\ \ \supp^{hyp}_\mbb{P}(M\ot V)=\supp^{hyp}_\mbb{P}(M)\cap\supp^{hyp}_\mbb{P}(V),
\end{array}
\end{equation}
whenever $V$ is finite-dimensional and centralizes the simples.
\par

When $\msf{u}$ is in particular geometrically Chevalley, and all hypersurfaces for a given Chevalley integration satisfy the thick subcategory lemma, then the equalities \eqref{eq:603} hold for arbitrary $M$ and arbitrary finite-dimensional $V$.
\end{theorem}

\begin{proof}
By changing base we may take $K=k$.  Fix $I=\operatorname{Irrep}(\Lambda)$.  By Lemma \ref{lem:weak} there is an inclusion $\supp^{hyp}_\mbb{P}(V\ot M)\subset \supp^{hyp}_\mbb{P}(V)\cap\supp^{hyp}_\mbb{P}(M)$.  So we need only provide the opposite inclusion.  Suppose, to this end, that $V\ot M$ is of finite projective dimension over $\msf{U}_c$ while $V$ is \emph{not} of finite projective dimension over $\msf{U}_c$.  Then
\[
k\in \langle \lambda\ot V\ot\mu:\lambda,\mu\in I\rangle\subset D_{coh}(\msf{U}_c),
\]
by the thick subcategory lemma.  The centralizing hypothesis simplifies the expression above to give $k\in \langle \lambda\ot V:\lambda\in I\rangle$.  Apply the endofunctor $-\ot M$ of $D^b(\msf{U}_c)$ to find
\[
M\in \langle \lambda\ot V\ot M:\lambda\in I\rangle\subset \langle \Proj(\msf{U}_c)\rangle,
\]
so that $M$ is seen to be of finite projective dimension.
\par

We have now shown that if a point $c$ is in the intersection $\supp^{hyp}_\mbb{P}(V)\cap \supp^{hyp}_\mbb{P}(M)$ then $c$ is also in $\supp^{hyp}_\mbb{P}(V\ot M)$.  This establishes the inclusion $\supp^{hyp}_\mbb{P}(V)\cap\supp^{hyp}_\mbb{P}(M)\subset \supp^{hyp}_\mbb{P}(V\ot M)$ and subsequent equality.  One argues similarly to obtain the desired formula for the support of the product $M\ot V$, in this case employing the equality $\langle \lambda\ot V\ot \mu:\lambda,\mu\in I\rangle=\langle V\ot \lambda:\lambda\in I\rangle$.  We note, finally, that in the geometrically Chevalley case, all objects centralize the simples by Lemma \ref{lem:w/e}, so that the identities \eqref{eq:603} hold globally.
\end{proof}

We have a local version of the above theorem, which exhibits a certain sidedness.

\begin{theorem}\label{thm:tpp_loc}
Consider local $\msf{u}$ with local integration $\msf{U}\to \msf{u}$.  Suppose that all hypersurfaces for $\msf{U}$ satisfy the thick subcategory lemma (see \ref{sect:tsl}).  Then hypersurface support for $\Stab(\msf{u})$ satisfies the tensor product property
\begin{equation}\label{eq:604}
\supp^{hyp}_\mbb{P}(M\ot V)=\supp^{hyp}_\mbb{P}(M)\cap\supp^{hyp}_\mbb{P}(V),
\end{equation}
where $M$ is arbitrary and $V$ is finite-dimensional.
\end{theorem}

\begin{proof}
The same as the proof of Theorem \ref{thm:tpp}, except we note that only the left action $M\ot-:D^b(\msf{U}_c)\to D^b(\msf{U}_c)$ is well-defined when $\msf{U}\to \msf{u}$ is not conormal.
\end{proof}

Applying Theorems \ref{thm:tpp} and \ref{thm:tpp_loc}, in conjunction with Proposition \ref{prop:351} and Proposition \ref{prop:538}, provides

\begin{corollary}\label{cor:examples}
For $\msf{u}$ among {\rm (\hyperref[it:F1]{F1})--(\hyperref[it:F3]{F4})}, hypersurface support $(\kappa(\msf{Y}),\supp^{hyp}_\mbb{P})$ provides a lavish support theory for the stable category $\stab(\msf{u})$.
\end{corollary}

\begin{proof}
In \cite{negronpevtsova} it was shown that hypersurface support for the small stable category, in these cases, is multiplicative.  The fact that hypersurface support is exhaustive follows from Proposition \ref{prop:538} and the fact that cohomological support is exhaustive.  Theorems \ref{thm:tpp} and \ref{thm:tpp}, in conjunction with \cite[Theorem 6.1]{negronpevtsova2}, imply that hypersurface support for the big stable category provides a faithful tensor extension of its compact restriction to all of $\Stab(\msf{u})$.
\end{proof}

By Corollary \ref{cor:examples} we find further that, when $\kappa:\msf{Y}\to \mbb{P}(m_Z/m_Z^2)$ is a closed embedding, hypersurface support for $\Stab(\msf{u})$ provides a faithful tensor extension of \emph{cohomological} support to the big stable category.

\begin{corollary}
Cohomological support is a lavish support theory for $\msf{u}$ a bosonized quantum complete intersection, or the algebra of functions on an infinitesimal (finite) group scheme.
\par

Similarly, cohomological support is a lavish support theory for $\msf{u}$ a small quantum Borel algebra in type $A$, at large order $q$, or the double $\mcl{D}(B_{(1)})$ in large characteristic.
\end{corollary}

By large order $q$ we mean $\ord(q)>n+1$ in type $A_n$, and by large characteristic we mean $p>\dim B+1$.

\begin{proof}
In the cases listed above the map $\kappa$ was shown to be a closed embedding in \cite{negronpevtsova}.  See specifically \cite[Lemma 7.5]{negronpevtsova} and, for types (\hyperref[it:F1]{F1})--(\hyperref[it:F4]{F4}) respectively, \cite{berghoppermann08} \cite[\S 11.2]{negronpevtsova}, \cite{ginzburgkumar93}, \cite[Theorem 6.10]{friedlandernegron18}, and \cite[Lemma 10.4]{negronpevtsova}. 
\end{proof}

At Corollaries \ref{cor:948} and \ref{cor:949} below the restrictions on the order of $q$ for $u_q(B)$, and the characteristic of the base field for $\mcl{D}(B_{(1)})$, are shown to be superfluous.

\section{Thick ideals and spectra for $\stab$(\hyperref[it:F1]{F1})--$\stab$(\hyperref[it:F3]{F3})}
\label{sect:ideals_q}

We classify thick ideals, and discuss central generation of ideals, for the families (\hyperref[it:F1]{F1})--(\hyperref[it:F3]{F3}).  Throughout the section we use, implicitly, the relationship between cohomological and hypersurface support provided by Proposition \ref{prop:538}.  In particular, we take for granted that the cohomological and hypersurface supports for an integrable Hopf algebra agree whenever the map
\[
\kappa:\msf{Y}=\Proj\Ext_{\msf{u}}^\ast(k,k)_{\rm red}\to \mbb{P}(m_Z/m_Z^2)
\]
provided in Section \ref{sect:hyp} is injective, and hence topologically a closed embedding.

\subsection{Classification of ideals and the map $\kappa$}

Before we begin, let us give a general result which speaks to the nature of the map $\kappa$.

\begin{lemma}\label{lem:class_kap}
Suppose that a finite-dimensional Hopf algebra $\msf{u}$ is integrable, with some fixed integration, and suppose that the corresponding hypersurface support $(\kappa(\msf{Y}),\supp^{hyp}_\mbb{P})$ is multiplicative and classifies thick ideals in $\stab(\msf{u})$.  Suppose also that cohomological support for $\msf{u}$ satisfies
\[
\supp^{coh}_\msf{Y}(V\ot W)\subset \left(\supp^{coh}_\msf{Y}(V)\cap\supp^{coh}_\msf{Y}(W)\right),
\]
for all $V$ and $W$ in $\stab(\msf{u})$.  Then the map $\kappa:\msf{Y}\to\mbb{P}$ is topologically a closed embedding.
\end{lemma}

\begin{proof}
Indeed, suppose this is not the case, and that $\kappa$ is not a closed embedding.  Then there are two closed points $x, x'\in \msf{Y}$ with $\kappa(x)=\kappa(x')$ while $x\neq x'$.  Since all closed subsets in $\msf{Y}$ are realized as the support of some object, it follows that the thick ideals $\msc{I}_x$ and $\msc{I}_{x'}$ of objects supported at $x$ and $x'$ are distinct.  But the hypersurface supports of the ideals $\msc{I}_x$ and $\msc{I}_{x'}$ agree in this case, and we reach a contradiction.  So we see that $\kappa$ must be a closed embedding, as claimed.
\end{proof}

\subsection{(\hyperref[it:F1]{F1}) Quantum complete intersections}

We consider a standard quantum complete intersection $\msf{a}_q=\msf{a}_q(P)$ with its usual integration $\msf{A}_q\to \msf{a}_q$ and corresponding hypersurface support $\supp^{hyp}_\mbb{P}$. 
\par

In this setting the map $\kappa:\msf{Y}\to \mbb{P}(m_Z/m_Z^2)$ is an isomorphism, and the cohomological and hypersurface supports for $\stab(\msf{u})$ agree \cite{berghoppermann08} \cite[\S 11.2]{negronpevtsova}.  The dimension of $m_Z/m_Z^2$ is equal to the number of generators of $\msf{a}_q$, which we fix as $n$, and we have
\[
\mbb{P}^{n-1}=\msf{Y}\cong \mbb{P}(m_Z/m_Z^2).
\]

\begin{theorem}\label{thm:qci_ideals}
Thick ideals in $\stab(\msf{a}_q)$ are classified by cohomological support $(\mbb{P}^{n-1},\supp_\mbb{P}^{coh})$, and there is a homeomorphism
\[
\mbb{P}^{n-1}\overset{\cong}\to\cSpec(\stab(\msf{a}_q)).
\]
\end{theorem}

\begin{proof}
By Lemma \ref{lem:weak} and Corollary \ref{cor:examples}, cohomological support for $\stab(\msf{a}_q)$ admits a tensor extension to $\Stab(\msf{a}_q)$.  So the result follows by Theorem \ref{thm:classify} and Theorem \ref{thm:chevalley_balmer}.
\end{proof}

\subsection{(\hyperref[it:F2]{F2}) The quantum Borel in type $A$}

We consider the quantum Borel $u_q(B)$ with its associated De Concini-Kac integration $U^{DK}_q(B)$, and corresponding hypersurface support.  At sufficiently large order $q$, the map $\kappa:\mbb{P}(\mfk{n})\cong\msf{Y}\to \mbb{P}(m_Z/m_Z^2)$ is an isomorphism, although at low order $q$ this is not a priori known to be true.  Throughout this subsection we take
\[
\msf{Y}_q:=\msf{Y}=\Proj\Ext^\ast_{u_q(B)}(\mbb{C},\mbb{C})_{\rm red}\ \ \text{and}\ \ \supp^{coh}_q:=\supp^{coh}_\msf{Y},
\]
for the sake of specificity.

As we saw in the proof of Theorem \ref{thm:qci_ideals} above, Lemma \ref{lem:weak} and Corollary \ref{cor:examples} imply

\begin{proposition}\label{prop:borel_ideals}
Consider $\mbb{G}$ an almost-simple algebraic group in type $A$, and let $u_q(B)$ be the corresponding small quantum Borel, at arbitrary odd order parameter $q$.  Hypersurface support $(\kappa(\msf{Y}_q),\supp^{hyp}_\mbb{P})$ classifies thick ideals in $\stab(u_q(B))$.
\end{proposition}

By applying Lemma \ref{lem:class_kap}, we obtain the following somewhat surprising implication of Proposition \ref{prop:borel_ideals}.

\begin{corollary}\label{cor:948}
In type $A$, at arbitrary odd order $q$, the map $\kappa:\msf{Y}_q\to \mbb{P}(m_Z/m_Z^2)$ is topologically a closed embedding.
\end{corollary}

Since hypersurface and cohomological supports now agree at arbitrary $q$, we have

\begin{theorem}\label{thm:borel_ideals}
For $u_q(B)$ in type $A$, at arbitrary odd order $q$, and any $V$ and $W$ in $\rep(u_q(B))$, cohomological support satisfies the tensor product property
\begin{equation}\label{eq:1588}
\supp^{coh}_q(V\ot W)=\supp^{coh}_q(V)\cap\supp^{coh}_q(W).
\end{equation}
Furthermore, cohomological support $(\msf{Y}_q,\supp^{coh}_q)$ classifies thick ideals in $\stab(u_q(B))$, and there is a homeomorphism
\[
\msf{Y}_q\overset{\cong}\to \cSpec(\stab(u_q(B))).
\]
\end{theorem}

As remarked above, at $\ord(q)$ larger than the Coxeter number $h$ we have $\msf{Y}_q=\mbb{P}(\mfk{n})$ \cite{ginzburgkumar93}, so that thick ideals in $\stab(u_q(B))$ are classified more specificially by specialization closed subsets in $\mbb{P}(\mfk{n})$, and we have a homeomorphism
\[
\mbb{P}(\mfk{n})\overset{\cong}\to \cSpec(\stab(u_q(B))).
\]

\begin{remark}
The tensor product property \eqref{eq:1588} was already obtained at $\ord(q)>h$ in \cite{negronpevtsova}.  So we emphasize that the relation is now seen to hold at arbitrary odd order parameter $q$, in type $A$.
\end{remark}

\subsection{(\hyperref[it:F3]{F3}) The Drinfeld double of $B_{(1)}$}

For the sake of specificity, we adopt the notations
\[
\msf{Y}_{dd}:=\msf{Y}=\Proj\Ext^\ast_{\mcl{D}(B_{(1)})}(k,k)_{\rm red}\ \ \text{and}\ \ \supp^{coh}_{dd}=\supp^{coh}_\msf{Y}.
\]
The subscript $dd$ here indicates the ``Drinfeld double".  We again apply Lemma \ref{lem:weak} and Corollary \ref{cor:examples} to obtain

\begin{proposition}\label{prop:db_ideals}
Consider $B$ a Borel in an almost-simple algebraic group $\mbb{G}$, over a field of arbitrary characteristic.  Hypersurface support $(\kappa(\msf{Y}_{dd}),\supp^{hyp}_\mbb{P})$ for the double classifies thick ideals in the stable category $\stab(\mcl{D}B_{(1)})$.
\end{proposition}

We now observe

\begin{corollary}\label{cor:949}
In the setting of Proposition \ref{prop:db_ideals}, the map $\kappa:\msf{Y}_{dd}\to \mbb{P}(m_Z/m_Z^2)$ is topologically a closed embedding.
\end{corollary}

\begin{proof}
Apply Lemma \ref{lem:class_kap}.
\end{proof}

As a consequence we obtain a characteristic free description of support for the double $\mcl{D}(B_{(1)})$.

\begin{theorem}\label{thm:db_ideals}
Cohomological support $(\msf{Y}_{dd},\supp^{coh}_{dd})$ classifies thick ideals in $\stab(\mcl{D}B_{(1)})$, and satisfies the tensor product property
\[
\supp^{coh}_{dd}(V\ot W)=\supp^{coh}_{dd}(V)\cap\supp^{coh}_{dd}(W).
\]
There is furthermore a homeomorphism $\msf{Y}_{dd}\overset{\cong}\to \cSpec(\stab(\mcl{D}B_{(1)}))$.
\end{theorem}

As with the quantum Borel situation, the space $\msf{Y}_{dd}$ is completely unambiguous in large characteristic, where we have
\[
\msf{Y}_{dd}=\mbb{P}\big(\mfk{n}\times (\mfk{b}^\ast)^{(1)}\big)
\]
\cite{friedlandernegron18}.  Here $\mfk{b}$ is the Lie algebra for $B$, $\mfk{n}$ is the nilpotent radical in $\mfk{b}$, and $\mfk{b}^{(1)}$ denotes the Frobenius twist.

\subsection{A general result for central generation}

Below we prove that all thick ideals in the stable categories for (\hyperref[it:F1]{F1})--(\hyperref[it:F3]{F3}) are centrally generated.  We rely on the following general observation.

\begin{lemma}\label{lem:general}
Consider $\msc{C}$ any finite tensor category, and suppose that thick ideals in the stable category $\stab(\msc{C})$ are classified by cohomological support.  Suppose additionally that the forgetful functor $F:Z(\msc{C})\to \msc{C}$ induces a topological closed embedding
\[
F^\ast:\Spec\Ext^\ast_{\msc{C}}(\1,\1)\to \Spec\Ext^\ast_{Z(\msc{C})}(\1,\1),
\]
and that all of the cohomology rings in question are finitely generated algebras.  Then all thick ideals in $\stab(\msc{C})$ are centrally generated.
\end{lemma}

\begin{proof}
Fix $\msc{Z}=Z(\msc{C})$.  For an extension $\xi:\1\to \Sigma^n\1$ in $D^b(\msc{Z})$, we consider the mapping cone $L_\xi=\operatorname{cone}(\xi)$.  The cohomological support of this object is the vanishing locus $Van(\xi)$ of $\xi\in \Ext^n_\msc{Z}(\1,\1)$ in the spectrum of cohomology.  Furthermore, the support of a product of such complexes $L_{\xi_1}\ot\dots \ot L_{\xi_n}$ is the vanishing locus $Van(\xi_1,\dots,\xi_n)$ \cite[Theorem 5.2]{berghplavnikwitherspoon}.  So in particular, all closed subsets in the spectrum of cohomology are realized as supports of products of mapping cones $L_\xi$.
\par

Now consider the forgetful (tensor) functor $F:\msc{Z}\to \msc{C}$, and its derived counterpart, which we denote by the same letter $F:D^b(\msc{Z})\to D^b(\msc{C})$.   We have $F(L_\xi)\cong L_{F(\xi)}$ and also $F(L_1\ot\dots \ot L_n)\cong F(L_1)\ot\dots \ot F(L_n)$.  Since all closed subvarieties in $\Proj\Ext^\ast_\msc{C}(\1,\1)$ are cut out by functions in the image of the map $F:\Ext^\ast_\msc{Z}(\1,\1)\to \Ext^\ast_\msc{C}(\1,\1)$, by hypothesis, the above information implies that all such closed subvarieties are realized as the supports of objects of the form $F(V)$, with $V$ in $\msc{Z}$.  Similarly, all specialization closed subset in $\Proj\Ext^\ast_\msc{C}(\1,\1)$ are realized as supports of collections of objects in the image of the functor $F:\msc{Z}\to \msc{C}$.
\par

Since thick ideals in $\stab(\msc{C})$ are classified by cohomological support, by hypothesis, we see now that all thick ideals in $\stab(\msc{C})$ admit a central generating set.
\end{proof}

\subsection{Central generation of ideals}
\label{sect:central}

Since $\rep(\mcl{D}(B_{(1)}))=Z(\rep B_{(1)})$ is already braided, all thick ideals in its stable category are immediately centrally generated.  (Note that in this case, the center $Z(\rep B_{(1)})$ is our initial category $\msc{C}$, so that the second iteration $Z(Z(\rep B_{(1)}))$ would be $Z(\msc{C})$.)  In the cases of $\msf{a}_q$ and $u_q(B)$, however, the representation categories are not braided in general.  Despite this fact, we see below that all thick ideals in their respective stable categories remain centrally generated.

\begin{lemma}\label{lem:1018}
The maps $\Spec\Ext^\ast_{\msf{a}_q}(\mbb{C},\mbb{C})\to \Spec\Ext^\ast_{\mcl{D}(\msf{a}_q)}(\mbb{C},\mbb{C})$ for the quantum complete intersection, and $\Spec\Ext^\ast_{u_q(B)}(\mbb{C},\mbb{C})\to \Spec\Ext^\ast_{\mcl{D}(u_q(B))}(\mbb{C},\mbb{C})$ for the quantum Borel in type $A$, are topologically closed embeddings.
\end{lemma}

\begin{proof}
We deal with the case of $\msf{a}_q$.  Let $Z_q\subset \msf{A}_q$ denote the parametrizing subalgebra for the given deformation $\msf{A}_q$.
\par

The dual $\msf{a}_q^\ast$ is of the form $\msf{a}_{\zeta}$ so that the Drinfeld double $\mcl{D}(\msf{a}_q)$ is a cocycle deformation of the product $\mcl{D}(\msf{a}_q)=(\msf{a}_q\ot\msf{a}_\zeta)_\sigma$ \cite{doitakeuchi94}.  Pulling back this cocycle along the integration $\msf{A}_q\ot\msf{A}_\zeta\to \msf{a}_q\ot\msf{a}_\zeta$, we can cocycle deform the product $\msf{A}_q\ot \msf{A}_\zeta$ to produce a deformation $\msf{D}=(\msf{A}_q\ot \msf{A}_\zeta)_\sigma$ parameterized by the (completed) product $Z_q\ot Z_\zeta$.  We then have a map of deformations
\[
\xymatrix{
Z_q\ot Z_\zeta\ar[r] & \msf{D}\ar[r] & \mcl{D}(\msf{a}_q)\\
Z_q\ar[u]\ar[r] & \msf{A}_q\ar[u]\ar[r] & \msf{a}_q.\ar[u]
}
\]
\par

One can use the above diagram to verify that the corresponding actions on cohomology fit into a diagram
\[
\xymatrix{
A_{Z_q\ot Z_\zeta}\ar@{->>}[d]\ar[r] & \Ext^\ast_{\mcl{D}(\msf{a})}(\mbb{C},\mbb{C})\ar[d]^{\res}\\
A_{Z_q}\ar[r] & \Ext^\ast_{\msf{a}_q}(\mbb{C},\mbb{C}).
}
\]
(One can establish commutativity of the above diagram by considering the description of the action on cohomology via the functor $D^b(B_Z)\to D^b(\msf{u})$ associated to a given deformation \cite{bezrukavnikovginzburg07} \cite[\S 3.4]{negronpevtsova}.)  Since the bottom map induces a topologically closed embedding on spectra, the restriction map from the cohomology of the double $\mcl{D}(\msf{a}_q)$ induces a topologically closed embedding on spectra of cohomology.
\par

The argument for the quantum Borel is completely similar, where we make an additional reference to Corollary \ref{cor:948} to observe that the map $A_Z\to \Ext^\ast_{u_q(B)}(\mbb{C},\mbb{C})$ induces a topologically closed embedding on spectra. 
\end{proof}

\begin{remark}
At large order $q$, it is well-known that the map in question for the quantum Borel is surjective, in arbitrary Dynkin type.  Indeed, the map on spectra of cohomology is identified with the closed embedding $\mfk{n}\to \mcl{N}$ of the positive nilpotent subalgebra in $\mfk{g}=\operatorname{Lie}(\mbb{G})$ into the nilpotent cone.
\end{remark}

\begin{theorem}\label{thm:center_genI}
Consider $\msc{C}$ the category of representations for a bosonized quantum complete intersection $\msf{a}_q$, or a quantum Borel $u_q(B)$ in type $A$.  All thick ideals in $\stab(\msc{C})$ are centrally generated.
\end{theorem}

\begin{proof}
Apply Lemma \ref{lem:general} and Lemma \ref{lem:1018}.
\end{proof}

\subsection{Remarks on the quantum Borel and relations with \cite{nakanovashawyakimov,nakanovashawyakimovII}}
\label{sect:rem}

In \cite{nakanovashawyakimov} the authors propose a classification of thick ideals for the quantum Borel $u_q(B)$ in arbitrary Dynkin type.  The arguments of \cite{nakanovashawyakimov} can be paraphrased as follows: Consider the triangular extension $\supp^{LC}$ of cohomological support to all of $\Stab(u_q(B))$ provided by local cohomology functors \cite{bensoniyengarkrause08}.  In work of Boe-Kujawa-Nakano \cite{boekujawanakano} this extension is proposed to have the following vanishing property:
\begin{equation}\label{eq:bkn}
\supp^{LC}(V\ot M)=\emptyset\ \Rightarrow\ \supp^{LC}(V)\cap\supp^{LC}(M)=\emptyset,
\end{equation}
for finite-dimensional $V$ and arbitrary $M$ \cite[Theorems 6.2.1 \& 6.5.1]{boekujawanakano}.  As one can see in the proof of Proposition \ref{prop:classify}, and as is argued directly in \cite{nakanovashawyakimov}, the implication \eqref{eq:bkn} can be used to conclude that cohomological support for the quantum Borel $u_q(B)$--now in arbitrary Dynkin type--classifies thick ideals in the stable category.  One subsequently identifies the prime ideal spectrum, and can use this identification to conclude that cohomological support satisfies the tensor product property, as is done in \cite{nakanovashawyakimovII}.
\par

While we have no objections to the general theory developed in \cite{nakanovashawyakimov,nakanovashawyakimovII}, and agree that the proposed classification of ideals is obtainable from the implication \eqref{eq:bkn}, we were unable to follow or reproduce a number of arguments from \cite{boekujawanakano}.  Indeed, the current project, along with its predecessor \cite{negronpevtsova}, was undertaken in order to provide a foundation for studies of the quantum group $u_q(\mbb{G})$ which is independent of \cite{boekujawanakano}.  So, in the present document, \emph{we} do not claim a classification of thick ideals for the quantum Borel in arbitrary Dynkin type.  We would encourage, however, all readers to give the paper \cite{boekujawanakano} proper consideration.

\begin{remark}
There is an additional claim from \cite{boekujawanakano,nakanovashawyakimovII} that the analysis of support for the quantum Borel is not sensitive to the choice of grouplikes, and hence not sensitive to centralizing hypotheses on objects.  This is in contrast to the specificity of the analysis given here, as in Section \ref{sect:geom_chev}, Lemma \ref{lem:1043}, Theorem \ref{thm:tpp}.  They arrive at this conclusion because the formula \eqref{eq:bkn} for support is proposed to be independent of the choice of grouplikes.  We make no analogous claim here, and employ in essential ways the centralization hypotheses provided by Lemma \ref{lem:w/e} (see Theorem \ref{thm:tpp} above).
\par

Let us comment, finally, that we have been in contact with the authors of the works \cite{boekujawanakano,nakanovashawyakimov,nakanovashawyakimovII}, and we thank them for continued productive discussions.
\end{remark}

\section{Hypersurface extension for $\Coh(\mcl{G})$}
\label{sect:hyper_G}

Fix now $k$ an \emph{algebraically closed} field.  We provide an analysis of hypersurface support for the algebra of functions $\O(\mcl{G})$ on a (generally non-connected) finite group scheme over $k$.
\par

As in the introduction, we let $\Coh(\mcl{G})=\rep(\O(\mcl{G}))$ denote the category of sheaves on $\mcl{G}$ with tensor structure induced by the group structure on $\mcl{G}$, or rather the Hopf structure on the algebra of functions.  Thus $\QCoh(\mcl{G})$ is identified with the category of arbitrary $\O(\mcl{G})$-modules, with its corresponding monoidal structure.
\par

For coherent sheaves on a general non-connected group scheme $\mcl{G}$, hypersurface support is not multiplicative.  One can see for example the results of \cite{bensonwitherspoon14,plavnikwitherspoon18}, or \cite[Example 10.2]{negronpevtsova}.  However, a sufficiently strong understanding of hypersurface support in this case can still be employed as a foundation for understanding thick ideals in the stable category $\stab(\Coh(\mcl{G}))$.  We use the results of the present section to classify ideals, and identify the prime ideal spectrum for $\Coh(\mcl{G})$ in Section \ref{sect:ideals_G}.

\subsection{Preliminary discussion}
Consider $\mcl{G}$ a finite group scheme, with identity component $\mcl{G}_o$ and subgroup of connected components $\pi=\pi_0(\mcl{G})$.
\par

We take a very specific approach to hypersurface support for $\Coh(\mcl{G})$.  First we note that, for the identity component $\mcl{G}_o$, we understand support completely.  This is because the algebra $\O(\mcl{G}_o)$ is local, and the results of Section \ref{sect:tensor} apply immediately.  Second, we note that sheaves on any other component $\mcl{G}_\lambda$ are identified with sheaves on $\mcl{G}_o$ via the translation action of the corresponding irreducible $\lambda\in \pi$.  So, the behaviors of the category $\Coh(\mcl{G})$ can be understood via an understanding of the connected component $\Coh(\mcl{G}_o)$, and an understanding of the action of the subgroup $\pi$ on $\Coh(\mcl{G})$.
\par

We apply this philosophy to produce (define), and analyze, \emph{a} hypersurface support for sheaves on a generally non-connected finite group scheme.

\begin{remark}
Our approach here deviates from that of the predecessor \cite{negronpevtsova}.
\end{remark}

\subsection{Generalities for $\Coh(\mcl{G})$}
\label{sect:G_gen}

Take $\mcl{G}$, $\mcl{G}_o$, and $\pi=\pi_0(\mcl{G})$ as above, and consider $k$ of odd characteristic $p$.   We cover some generalities for $\Coh(\mcl{G})$, which are not strictly necessary at this point, but which may help orient the reader.
\par

The algebra of functions on the connected component $\mcl{G}_o$ has the form $\O(\mcl{G}_o)=k[x_1,\dots,x_n]/(x_1^{p^{r_1}},\dots, x_n^{p^{r_n}})$ (see, for example, \cite[14.4]{waterhouse}).  This implies that the cohomology ring $\Ext^\ast_{\Coh(\mcl{G})}(k,k)=\Ext^\ast_{\Coh(\mcl{G}_0)}(k,k)$ is a tensor product of an exterior algebra, generated in degree $1$, with a polynomial ring, generated in degree $2$.  We take 
\[
\mcl{L}=(\Ext^2_{\Coh(\mcl{G})}(k,k)_{\rm red})^\ast,
\]
so that the projective spectrum of cohomology is $\msf{Y}=\mbb{P}(\mcl{L})$.
\par

The vector space $\mcl{L}$ can be seen as a twisting of the Lie algebra for $\mcl{G}$.  For example, when $\mcl{G}$ is the $r$-th Frobenius kernel $\mbb{G}_{(r)}$ in a smooth algebraic group $\mbb{G}$, then $\mcl{L}$ is naturally identified with $\operatorname{Lie}(\mcl{G})^{(r)}$.  See for example \cite[Proposition 3.5]{friedlandernegron18}.
\par

We should be clear that, although $\msf{Y}$ is a projective space in this case, the map $\kappa:\msf{Y}\to \mbb{P}(m_Z/m_Z^2)$ induced by a chosen integration $\O\to \O(\mcl{G}_o)$ for the identity component is \emph{not} necessarily an isomorphism.  It is, however, always a closed embedding \cite[Lemma 7.5]{negronpevtsova}.  Indeed, if $\mcl{G}_o\to \mcl{H}$ is an embedding into a connected algebraic group from which we deduce our integration $\O=\O(\mcl{H})\to \O(\mcl{G}_o)$, then $\kappa$ is an isomorphism if and only if $\dim\operatorname{Lie}(\mcl{H})=\dim\operatorname{Lie}(\mcl{G})$.

\subsection{Characteristic 2}
\label{sect:char2}

As usual, characteristic $2$ is somewhat special.  Fix $k$ algebraically closed of characteristic $2$, and let $\mcl G$ be any finite group scheme over $k$ with the identity component $\mcl{G}_o$. As  $\O(\mcl{G}_o)$ is a truncated polynomial algebra, just as in the odd characteristic case, the cohomology ring $\Ext^\ast_{\Coh(\mcl{G}_o)}(k,k)$ is the tensor product of a polynomial ring generated in degrees $1$ and $2$ with an exterior algebra generated in degree $1$ \cite[3.3]{evens91}.  The reduced ring $\Ext^\ast_{\Coh(\mcl{G})}(k,k)_{\rm red}=\Ext^\ast_{\Coh(\mcl{G}_o)}(k,k)_{\rm red}$ is then a polynomial ring generated in degrees $1$ and $2$.  Instead of taking $\mcl{L}$ to be a vector space here, which we identify with an affine $k$-scheme, we take simply
\[
\mcl{L}=\Spec\Ext^\ast_{\Coh(\mcl{G}_o)}(k,k)_{\rm red}
\]
so that again $\msf{Y}=\mbb{P}(\mcl{L})$.  We consider the scheme map $\kappa:\msf{Y}=\mbb{P}(\mcl{L})\to \mbb{P}(m_Z/m_Z^2)$ provided by an integration $\O\to \O(\mcl{G}_o)$ of the identity component.

\begin{lemma}
The map $\kappa:\mbb{P}(\mcl{L})\to \mbb{P}(m_Z/m_Z^2)$, provided by a choice of integration $\O\to \O(\mcl{G}_o)$, is topologically a closed embedding.
\end{lemma}

\begin{proof}
This follows by Corollary \ref{cor:examples} and Lemma \ref{lem:class_kap}, applied to $\msf{u}=\O(\mcl{G}_o)$.
\end{proof}

As a sanity check, consider the case where $\mcl{G}$ is the first Frobenius kernel $\mbb{G}_{(1)}$ in a smooth algebraic group $\mbb{G}$.  We then have $E=\Ext^\ast_{\mcl{G}}(k,k)=k[y_1,\dots, y_n]$, with the $y_i$ of degree $1$.  For the natural integration $\O(\mbb{G})\to \O(\mcl{G})$ provided by the embedding $\mcl{G}\to \mbb{G}$, the deformation map $A_Z\to E$ is an isomorphism onto the subalgebra $E'=k[y_1^2,\dots, y_n^2]$.  The above lemma now claims that the map on spectra
\[
\Spec E\to \Spec E',
\]
which is closed and surjective, is a homeomorphism.  It suffices to show that this map is injective on closed points.  However, such injectivity just follows from the fact that there is a unique solution to any equation $y^2-c=0$, $c\in k$, over algebraically closed field in characteristic 2.

\subsection{Hypersurface support for non-connected group schemes}

Consider $\mcl{G}$ a non-connected finite group scheme, with identity component $\mcl{G}_o$ and group of connected components $\pi\subset \mcl{G}$.  We define a support theory $(\msf{Y},\supp^{hyp}_\mbb{P})$ for $\mcl{G}$ as follows: Consider an integration $\O\to \O(\mcl{G}_o)$ as in (\hyperref[it:F4]{F4}), with its associated hypersurface support, and take for $M$ in $\QCoh(\mcl{G})$
\begin{equation}\label{eq:suppG}
\supp^{hyp}_\mbb{P}(M):=\bigcup_{\lambda\in \pi}\supp^{hyp}_\mbb{P}(M\ot\lambda\!\ |_{\mcl{G}_o}),
\end{equation}
where in the right hand expression $\supp^{hyp}_\mbb{P}$ denotes the usual hypersurface support for $\mcl{G}_o$, and $(-)|_{\mcl{G}_o}$ is the restriction functor to the open subscheme $\mcl{G}_o\subset \mcl{G}$.  We simply refer to $\supp^{hyp}_\mbb{P}$ as the hypersurface support for $\mcl{G}$, defined relative to a chosen integration $\O\to \O(\mcl{G}_o)$ of the identity component.
\par

To unravel this definition, any $M$ over $\mcl{G}$ is a sum of sheaves supported on the various components $\mcl{G}_\lambda=\mcl{G}_o\cdot \lambda$.  We may adopt an expression $M\cong \oplus_{\lambda\in \pi}M_\lambda\ot \lambda^{-1}$ for $M_\lambda=M\ot \lambda|_{\mcl{G}_o}$.  The support of $M$ over $\mcl{G}$ is then the union of the supports of its various components $M_\lambda$ over $\mcl{G}_o$.  We have the following basic observation.

\begin{lemma}\label{lem:837}
For $V$ in $\Coh(\mcl{G})$ hypersurface support $\supp^{hyp}_\mbb{P}(V)$ agrees with the usual cohomological support,
\[
\supp^{hyp}_\mbb{P}(V)=\Supp_\msf{Y}\Ext^\ast_{\Coh(\mcl{G})}(k,\oplus_{\lambda\in \pi}V\ot\lambda)^\sim.
\]
\end{lemma}

Note that we are adopting a ``right-handed" interpretation of cohomological support here.

\begin{proof}
This follows from the identification $\Ext^\ast_{\Coh(\mcl{G})}(k,-)=\Ext^\ast_{\Coh(\mcl{G}_o)}(k,-|_{\mcl{G}_o})$ and the fact that the map $\kappa:\msf{Y}\to \mbb{P}(m_Z/m_Z^2)$ is a closed embedding in this case \cite[Lemma 10.4]{negronpevtsova}, so that hypersurface support for the connected component $\Coh(\mcl{G}_o)$ agrees with cohomological support.
\end{proof}

As projectivity of $M$ is equivalent to projectivity of the collective summands $M_\lambda$, the detection theorem, Theorem \ref{thm:detec}, for $\mcl{G}_o$ implies that the hypersurface support detects projectivity for $\mcl G$. 

\begin{theorem}\label{thm:detectG}
For $\mcl{G}$ any finite group scheme, and $M$ in $\QCoh(\mcl{G})$, $\supp^{hyp}_\mbb{P}(M)=\emptyset$ if and ony if $M$ is projective in $\QCoh(\mcl{G})$, or equivalently vanishes in the stable category.
\end{theorem}

The tensor product property for $\QCoh(\mcl{G}_o)$, expressed in Theorem \ref{thm:tpp}, provides the following in the non-connected context.

\begin{theorem}[{cf.\ \cite[Corollary 10.8]{negronpevtsova}}]\label{thm:tppG}
For $\mcl{G}$ an arbitrary finite group scheme, $M$ in $\QCoh(\mcl{G})$, and $V$ centralizing the simples in $\Coh(\mcl{G})$, we have
\[
\supp^{hyp}_\mbb{P}(M\ot V)=\supp^{hyp}_\mbb{P}(M)\cap\supp^{hyp}_\mbb{P}(V).
\]
\end{theorem}

\begin{proof}
Consider an appropriate integration $\O\to \O(\mcl{G}_o)$ and any hypersurface $\O_c$.  By changing base we may assume that $c$ is a $k$-point.  We write $M\cong\oplus_\lambda M_\lambda\ot \lambda^{-1}$ and $V\cong\oplus_\lambda V_\lambda\ot\lambda^{-1}$.  We are checking the hypersurface support, over $\mcl{G}_o$, of the sum
\[
\oplus_{\lambda,\mu} (M_\lambda\ot \lambda^{-1}\ot V_{\mu\lambda^{-1}}\ot \lambda\mu^{-1})\ot\mu\cong \oplus_{\lambda,\mu} M_\lambda\ot \lambda^{-1}\ot V_\mu\ot \lambda.
\]
The centralizing hypotheses on $V$ imply that the above sum is isomorphic to
\[
\oplus_{\lambda,\mu}M_\lambda\ot V_{\operatorname{Ad}_\lambda(\mu)}=\oplus_{\lambda,\mu}M_\lambda\ot V_\mu.
\]
Now, we apply the tensor product property for $\Coh(\mcl{G}_o)$, provided by Theorem \ref{thm:tpp}, to find
\[
\begin{array}{rl}
\supp^{hyp}_\mbb{P}(\oplus_{\lambda,\mu}M_\lambda\ot V_\mu)& = \cup_{\lambda,\mu}\big(\supp^{hyp}_\mbb{P}(M_\lambda)\cap\supp^{hyp}_\mbb{P}(V_\mu)\big)\\
& = \big(\cup_\lambda\supp(M_\lambda)\big)\cap\big(\cup_\mu\supp^{hyp}_\mbb{P}(V_\mu)\big)\\
&=\supp^{hyp}_\mbb{P}(M)\cap\supp^{hyp}_\mbb{P}(V).
\end{array}
\]
\end{proof}

\section{Thick ideals and spectra for stable $\Coh(\mcl{G})$}
\label{sect:ideals_G}

We continue our analysis of thick ideals and support for sheaves $\Coh(\mcl{G})$ on a finite group scheme $\mcl{G}$.  We adopt categorical notations for the stable categories $\stab(\Coh(\mcl{G}))$ and $\Stab(\Coh(\mcl{G}))$, and take also
\[
\cSpec(\Coh(\mcl{G})):=\cSpec(\stab(\Coh(\mcl{G})))
\]
(see Definition \ref{def:prime}).  We maintain our assumption $k=\bar{k}$ through subsection \ref{sect:k=bark} below.  As explained in Remark \ref{rem:comm}, thick ideals in the stable category for $\Coh(\mcl{G})$ at \emph{connected} $\mcl{G}$ are classified via work of Stevenson \cite{stevenson13}.  So our main contributions, and our primary examples of interest, are to categories of sheaves on non-connected $\mcl{G}$.
\par

At the conclusion of the section we discuss central generation of ideals in the stable category for $\Coh(\mcl{G})$.  We note that these central generation questions are not reducible to earlier works in the field, even for connected $\mcl{G}$.

\subsection{Preliminary discussion}
Consider $\mcl{G}$ a finite group scheme over an algebraically closed base field $k$, with subgroup $\pi=\pi_0(\mcl{G})$ of connected components.  Take $\mcl{L}=(\Ext^2_{\Coh(\mcl{G})}(k,k)_{\rm red})^\ast$, or $\mcl{L}=\Spec\Ext^\ast_{\Coh(\mcl{G})}(k,k)_{\rm red}$ in characteristic $2$, so that $\msf{Y}=\mbb{P}(\mcl{L})$.
\par

As we have already discussed, in preamble to Section \ref{sect:hyper_G} for example, hypersurface support for $\Coh(\mcl{G})$ is generally ``bad".  Rather, cohomological support for $\Coh(\mcl{G})$ is known to not be multiplicative in general \cite{plavnikwitherspoon18}, and cohomological support is identified with hypersurface support via Lemma \ref{lem:837}.  The problem is that we have no containment
\[
\supp_\mbb{P}^{hyp}(V\ot W)\nsubseteq \supp_\mbb{P}^{hyp}(V)\cap\supp_\mbb{P}^{hyp}(W)
\]
in general.  This is equivalent to the failure of support to be stable under duality $\supp_\mbb{P}^{hyp}(V)\neq \supp_\mbb{P}^{hyp}(V^\ast)$.
\par

Having been presented with this tragic situation, we consider the quotient $p:\mbb{P}(\mcl{L})\to \mbb{P}(\mcl{L})/\!\sim$ of the projective space $\mbb{P}(\mcl{L})$ by the desired relations on closed subsets $\supp_\mbb{P}^{hyp}(V)\sim\supp_\mbb{P}^{hyp}(V^\ast)$.  (It turns out that the quotient $\mbb{P}(\mcl{L})/\!\sim$ is the quotient $\mbb{P}(\mcl{L})/\pi$ of $\mbb{P}(\mcl{L})$ by the adjoint action of the subgroup of connected components.)  We show that the resultant support theory $p\circ \supp_\mbb{P}$, which now takes values in this quotient $\mbb{P}(\mcl{L})/\sim$, is in fact multiplicative and classifies thick ideals in the stable category $\stab(\Coh(\mcl{G}))$.

\begin{remark}
Our presentation below does not follow, in a direct manner, the philosophy suggested above.  However, the constructions are equivalent.
\end{remark}

\subsection{Main result}

Take $\mcl{G}$, $\pi$, and $\mcl{L}$ as above.  The subgroup $\pi\subset \mcl{G}$ is identified with the group of irreducible objects in $\Coh(\mcl{G})$, and we have the corresponding adjoint action on the category
\[
\operatorname{Ad}_?:\pi\to \Aut^{\ot}(\Coh(\mcl{G})),\ \ \operatorname{Ad}_\lambda=\lambda\ot-\ot \lambda^{-1}.
\]
This action induces a (generally nontrivial) action on the projective spectrum of cohomology $\mbb{P}(\mcl{L})$.  We consider the quotient $\mbb{P}(\mcl{L})/\pi=_{top}\mbb{P}(\mcl{L}/\pi)$ via this adjoint action.
\par

Consider $p:\mbb{P}(\mcl{L})\to \mbb{P}(\mcl{L}/\pi)$ the quotient map.  Hypersurface support for $\Coh(\mcl{G})$, defined relative to a fixed integration $\O\to \O(\mcl{G}_o)$ of the connected component as in \eqref{eq:suppG}, now induces a support theory which takes values in $\mbb{P}(\mcl{L}/\pi)$.  We take specifically
\begin{equation}\label{eq:1847}
\supp_{\mcl{L}/\pi}:=p\circ\supp_{\mbb{P}}^{hyp}.
\end{equation}
We note that $\supp_{\mcl{L}/\pi}$ respects the triangulated structure on $\Stab(\Coh(\mcl{G}))$, in the precise sense of Section \ref{sect:supports}, as it inherits the necessary compatibilities from $\supp^{hyp}_\mbb{P}$.

\begin{remark}
The space $\mcl{L}/\pi$ is a conical variety which is not necessarily isomorphic to an affine space, and its projectivization $\mbb{P}(\mcl{L}/\pi)$ is therefore not necessarly a projective space.  We should also be clear that the theory $\supp_{\mcl{L}/\pi}$ still takes values in the \emph{projective} variety $\mbb{P}(\mcl{L}/\pi)$, not its affine counterpart $\mcl{L}/\pi$.  We omit the projectivization from the subscript for the sake of legibility.
\end{remark}

The point of this section is to prove the following result, which involves some relaxation of our hypotheses on the base field $k$.

\begin{theorem}\label{thm:G_classification}
Let $k$ be a perfect field, and $\mcl{G}$ be an arbitrary finite group scheme over $k$.  The support theory $(\mbb{P}(\mcl{L}/\pi),\supp_{\mcl{L}/\pi})$, defined as above, classifies thick ideals in $\stab(\Coh(\mcl{G}))$.  Furthermore, there is a homeomorphism
\[
\mbb{P}(\mcl{L}/\pi)\overset{\cong}\longrightarrow \cSpec(\Coh(\mcl{G})).
\]
\end{theorem}

Obviously, we've moved from an algebraically closed base field to a perfect base field is the statement of Theorem \ref{thm:G_classification}.  In the general (perfect) setting we take $\pi$ to be the \'etale subgroup $\pi:=\mcl{G}_{\rm red}$ in $\mcl{G}$ \cite[Ch 2, \S 5.2]{demazuregabriel70}, and have the adjoint action of $\pi$ on $\mcl{G}$, and hence on $\Coh(\mcl{G})$.  This generalized adjoint action induces an action of $\pi$ on the spectrum of cohomology, and we can again consider the quotient $\mbb{P}(\mcl{L}/\pi)$.  The support theory $(\mbb{P}(\mcl{L}/\pi),\supp_{\mcl{L}/\pi})$ referenced above, for non-algebraically closed $k$, is constructed explicitly in the proof of Theorem \ref{thm:G_classification} (Section \ref{sect:k}).

\subsection{The proof of Theorem \ref{thm:G_classification} when $k=\bar{k}$}
\label{sect:k=bark}

Let us maintain our assumption that $k=\bar{k}$, for the moment, and let $\mcl{G}$ be a finite group scheme over $k$.  Note that the category $\Coh(\mcl{G})$ is Chevalley, with subcategory of semisimple objects identified with the category of sheaves $\Coh(\pi)$ on $\pi=\mcl{G}_{\rm red}$.  We have the adjoint action of $\pi$ on $\Coh(\mcl{G})$, discussed above, and objects in $\Coh(\mcl{G})$ which centralize the simples are precisely $\pi$-equivariant objects in $\Coh(\mcl{G})$.  We let $\Coh(\mcl{G})^\pi$ denote the category of $\pi$-equivariant sheaves on $\mcl{G}$.
\par

The forgetful functor $Z^{\Coh(\pi)}(\Coh(\mcl{G}))=\Coh(\mcl{G})^\pi\to \Coh(\mcl{G})$ has a right adjoint $V\mapsto \oplus_{\lambda\in \pi}\operatorname{Ad}_\lambda(V)$ which sends any object to its orbit under the adjoint action of $\pi$.  Below we let $\lambda\cdot-:\mbb{P}(\mcl{L})\to \mbb{P}(\mcl{L})$ denote the action of $\lambda\in \pi$ on the projective spectrum of cohomology $\msf{Y}=\mbb{P}(\mcl{L})$ induced by the adjoint action of $\pi$ on $\Coh(\mcl{G})$.

\begin{lemma}\label{lem:1495}
For any $V$ in $\Coh(\pi)$, and $\lambda\in \pi$, we have an equality of supports $\supp_{\mcl{L}/\pi}(V)=\supp_{\mcl{L}/\pi}(\operatorname{Ad}_\lambda V)$.  Subsequently,  we have
\[
\supp_{\mcl{L}/\pi}(V)=\supp_{\mcl{L}/\pi}(\oplus_{\lambda\in \pi}\operatorname{Ad}_\lambda V).
\]
\end{lemma}

\begin{proof}
The second claim follows from the first and the general property $\supp_{\mcl{L}/\pi}(V\oplus W)=\supp_{\mcl{L}/\pi}(V)\cup\supp_{\mcl{L}/\pi}(W)$.  As for the first claim, we have the commuting diagram
\[
\xymatrix{
\Ext^\ast_{\Coh(\mcl{G})}(k,k)\ar[rr]^(.45){\operatorname{Ad}_\lambda}\ar[d]_{-\ot V} & & \Ext^\ast_{\Coh(\mcl{G})}(k,k)\ar[d]^{-\ot \operatorname{Ad}_\lambda V}\\
\Ext^\ast_{\Coh(\mcl{G})}(V,V)\ar[rr]^(.45){\operatorname{Ad}_\lambda} & & \Ext^\ast_{\Coh(\mcl{G})}(\operatorname{Ad}_\lambda V,\operatorname{Ad}_\lambda V)
}
\]
which implies that $\supp^{coh}_\msf{Y}(\operatorname{Ad}_\lambda V)=\lambda\cdot\supp^{coh}_\msf{Y}(V)$.  Since hypersurface support and cohomological support are identified, by Lemma \ref{lem:837}, we similarly have $\supp^{hyp}_{\mbb{P}}(\operatorname{Ad}_\lambda V)=\lambda\cdot\supp^{hyp}_{\mbb{P}}(V)$.  Project to the quotient to find
\[
\begin{array}{rl}
\supp_{\mcl{L}/\pi}(\operatorname{Ad}_\lambda V)& = p(\supp^{hyp}_\mbb{P}(\operatorname{Ad}_\lambda V))\\
& = p\left(\lambda\cdot\supp^{hyp}_\mbb{P}(V)\right)= p(\supp^{hyp}_\mbb{P}(V))=\supp_{\mcl{L}/\pi}(V),
\end{array}
\]
establishing the claimed equality.
\end{proof}

\begin{lemma}\label{lem:1102}
For any $M$ in $\QCoh(\mcl{G})$, and $V$ in $\Coh(\mcl{G})$, we have
\[
\supp_{\mcl{L}/\pi}(M\ot V)\subset \left(\supp_{\mcl{L}/\pi}(M)\cap\supp_{\mcl{L}/\pi}(V)\right).
\]
The above inclusion is an equality when $V$ centralizes the simples in $\Coh(\mcl{G})$.
\end{lemma}

\begin{proof}
As $V$ is in the thick subcategory generated by $\Lambda=\oplus_i \lambda_i$, we have
\[
\supp_{\mcl{L}/\pi}(M\ot V)\subset \supp_{\mcl{L}/\pi}(\oplus_i M\ot\lambda_i)=\supp_{\mcl{L}/\pi}(M),
\]
where the equality follows by the definition of hypersurface support for $\mcl{G}$ \eqref{eq:suppG}.  Now, since $M$ is in the localizing subcategory generated by $\Lambda$, we have the other inclusion
\begin{equation}\label{eq:1112}
\supp_{\mcl{L}/\pi}(M\ot V)\subset \supp_{\mcl{L}/\pi}(V)\ \ \text{\em provided}\ \ \supp_{\mcl{L}/\pi}(\Lambda\ot V)\subset \supp_{\mcl{L}/\pi}(V).
\end{equation}
By Lemma \ref{lem:1495}, and the fact that $\Lambda=\oplus_{\lambda\in \pi} \lambda$ itself centralizes the simples, we have
\[
\begin{array}{rl}
\supp_{\mcl{L}/\pi}(\Lambda\ot V) & = \supp_{\mcl{L}/\pi}\left(\oplus_\lambda \operatorname{Ad}_\lambda(\Lambda\ot V)\right)\\
& = \supp_{\mcl{L}/\pi}(\oplus_\lambda \Lambda\ot \operatorname{Ad}_\lambda(V)) = \supp_{\mcl{L}/\pi}(\oplus_\lambda \operatorname{Ad}_\lambda(V)\ot \Lambda).
\end{array}
\]
This final support space is equal to $\supp_{\mcl{L}/\pi}(\oplus_\lambda\operatorname{Ad}_\lambda V)=\supp_{\mcl{L}/\pi}(V)$, again by the definition \eqref{eq:suppG}.  So we obtain the right-hand containment of \eqref{eq:1112}, and thus the desired containment $\supp_{\mcl{L}/\pi}(M\ot V)\subset \supp_{\mcl{L}/\pi}(V)$.
\par

The claimed equality when $V$ is $\pi$-equivariant comes from the analogous equality for $\supp^{hyp}_\mbb{P}$, provided in Theorem \ref{thm:tppG}, and the fact that the closed subspace $\supp^{hyp}_\mbb{P}(V)$ is $\pi$-stable in this case.
\end{proof}

Detection for hypersurface support, Theorem \ref{thm:detectG}, also implies

\begin{lemma}\label{lem:1703}
For $M$ in $\QCoh(\mcl{G})$, $\supp_{\mcl{L}/\pi}(M)=\emptyset$ if and only if $M$ vanishes in the stable category $\Stab(\Coh(\mcl{G}))$.
\end{lemma}

At this point we essentially know that the support theory $(\mbb{P}(\mcl{L}/\pi),\supp_{\mcl{L}/\pi})$ is faithful, multiplicative, and provides a tensor extension to the big stable category $\Stab(\Coh(\mcl{G}))$.  So we obtain the main result of the section, at least when $k$ is algebraically closed.

\begin{proof}[Proof of Theorem \ref{thm:G_classification}, when $k=\bar{k}$]
The fact that cohomological support is exhaustive implies that it's pushforward along $p$ is also exhaustive.  Multiplicativity follows from Lemma \ref{lem:1102}, \cite[Lemma 10.3]{negronpevtsova}, and Lemma \ref{lem:1495}, after we note that the cohomological supports defined via the left and right actions of $\Ext^\ast(k,k)$ agree for objects which centralize the simples \cite[Proposition 5.7.1]{benson91} \cite[Proposition 2.10.8]{egno15}.  Finally, Lemmas \ref{lem:1102} and \ref{lem:1703} imply that the pair $(\mbb{P}(\mcl{L}/\pi),\supp_{\mcl{L}/\pi})$ provides a faithful tensor extension of the pushforward of cohomological support along the quotient map $p:\mbb{P}(\mcl{L})\to \mbb{P}(\mcl{L}/\pi)$ to the big stable category.  This is all to say, the pair $(\mbb{P}(\mcl{L}/\pi),\supp_{\mcl{L}/\pi})$ provides a lavish support theory for $\stab(\Coh(\mcl{G}))$.  Since $\Coh(\mcl{G})$ is Chevalley, we apply Theorems \ref{thm:classify} and \ref{thm:chevalley_balmer} to see that this support classifies thick ideals in $\stab(\Coh(\mcl{G}))$, and also calculates the spectrum.
\end{proof}

\subsection{Theorem \ref{thm:G_classification} over an arbitrary perfect base}
\label{sect:k}

We consider now $\mcl{G}$ an arbitrary finite group scheme over a perfect base field $k$.

\begin{proof}[Proof of Theorem \ref{thm:G_classification} over general $k$]
The fact that $\pi=\mcl{G}_{\rm red}\subset \mcl{G}$ is a subgroup in this case \cite[Ch 2, \S 5, Corollaire 2.3]{demazuregabriel70} implies that--and is equivalent to the fact that--$\Coh(\mcl{G})$ is Chevalley.  Consider a (finite) Galois extension $k\to K$ which splits the \'etale subgroup $\pi=\mcl{G}_{\rm red}$, so that every closed point in $\mcl{G}_K$ is a $K$-point.  Then $\pi_K$ is identified with a discrete group, and all simples in $\Coh(\mcl{G}_K)$ are $1$-dimensional.  We then define the support theory
\[
\supp_{\mcl{L}_K/\pi_K}:=p_K\circ \supp^{hyp}_{\mbb{P}_K}
\]
just as in the case $k=\bar{k}$, and follow the above arguments verbatim to find that the pair $(\mbb{P}(\mcl{L}_K/\pi_K),\supp^{hyp}_{\mcl{L}_K/\pi_K})$ provides a lavish support theory for $\stab(\Coh(\mcl{G}_K))$.
\par

We note that the Galois group $\Gamma=\operatorname{Gal}(K/k)$ acts on $\mbb{P}(\mcl{L}_K)$ and $\mbb{P}(\mcl{L}_K/\pi_K)$ and we have the diagram
\[
\xymatrix{
\mbb{P}(\mcl{L}_K)\ar[d]_{p_K}\ar[rr]^q & & \mbb{P}(\mcl{L})\ar[d]^p\\
\mbb{P}(\mcl{L}_K/\pi_K)\ar[rr]^{q'} & & \mbb{P}(\mcl{L}/\pi),
}
\]
where the horizontal maps are quotient maps by that action of $\Gamma$.  We define, for any $M$ in $\QCoh(\mcl{G})$, the $\mcl{P}(\mcl{L}/\pi)$-valued support
\begin{equation}\label{eq:1623}
\supp^{hyp}_{\mcl{L}/\pi}(M):=q'\circ \supp_{\mcl{L}_K/\pi_K}(M_K).
\end{equation}
One can check that $\supp^{hyp}_{\mcl{L}/\pi}(V)=p\circ\supp^{coh}_\msf{Y}(V)$ for any finite-dimensional $V$, so that this support theory is exhaustive.  Also, since the base change $M_K$ is projective if and only if $M$ is projective \cite[Corollary 3.2 \& Lemma 5.3]{negronpevtsova2}, and hypersurface support for $\Coh(\mcl{G}_K)$ is faithful by Lemma \ref{lem:1703}, the support theory \eqref{eq:1623} is also faithful.  We note that the base change map $(-)_K:\QCoh(\mcl{G})\to \QCoh(\mcl{G}_K)$ is monoidal and preserves objects which centralize the simples, and that the support $\supp_{\mcl{L}_K/\pi_K}(V_K)$ is $\Gamma$-stable for any $V$ in $\Coh(\mcl{G})$.  One can therefore recover multiplicativity of the theory $\supp_{\mcl{L}/\pi}$ from multiplicativity of the theory $\supp_{\mcl{L}_K/\pi_K}$.
\par

From the above information we see that the pair $(\mbb{P}(\mcl{L}/\pi),\supp_{\mcl{L}/\pi})$ provides a lavish support theory for $\stab(\Coh(\mcl{G}))$, and therefore classifies thick ideals, by Theorem \ref{thm:classify}.  We also apply Theorem \ref{thm:chevalley_balmer} to obtain the claimed calculation of the spectrum for $\stab(\Coh(\mcl{G}))$.
\end{proof}

We note that the above proof applies, more generally, to coherent sheaves $\Coh(\mcl{G})$ on any finite group scheme for which the reduced subscheme $\mcl{G}_{\rm red}\subset \mcl{G}$ is in fact a sub\emph{group}.

\subsection{Central generation}
\label{sect:cen_G}
We now consider central generation of ideals in the stable category $\stab(\Coh(\mcl{G}))$, or rather the possibility of such central generation.  We allow $k$ to be perfect throughout.

\begin{lemma}\label{lem:1190}
The map $\Ext^\ast_{Z(\Coh(\mcl{G}))}(k,k)\to \Ext^\ast_{\Coh(\mcl{G})}(k,k)$ induced by the forgetful functor $Z(\Coh(\mcl{G}))\to \Coh(\mcl{G})$ is a finite algebra map.
\end{lemma}

\begin{proof}
This follows by the main result of \cite{negron} and \cite[Proposition 3.3]{negronplavnik}.
\end{proof}

The above lemma ensures that we get a well-defined map on projective spectra
\[
\Proj\Ext^\ast_{\Coh(\mcl{G})}(k,k)\to \Proj\Ext^\ast_{Z\!\Coh(\mcl{G})}(k,k).
\]

\begin{theorem}\label{thm:center_genII}
Consider $\mcl{G}$ a finite group scheme with \'etale subgroup $\pi=\mcl{G}_{\rm red}$.  Suppose that the forgetful functor $Z(\Coh(\mcl{G}))\to \Coh(\mcl{G})$ induces a surjection onto the reduced, $\pi$-invariant cohomology $\Ext^\ast_{Z\!\Coh(\mcl{G})}(k,k)\to \Ext^\ast_{\Coh(\mcl{G})}(k,k)_{\rm red}^\pi$.  Then all thick ideals in $\stab(\Coh(\mcl{G}))$ are centrally generated.
\end{theorem}

We note, before beginning the proof, that the image of $\Ext^\ast_{Z\!\Coh(\mcl{G})}(k,k)$ in $\Ext^\ast_{\Coh(\mcl{G})}(k,k)$ does in fact lie in the invariants $\Ext^\ast_{\Coh(\mcl{G})}(k,k)^\pi$, under the adjoint action.  This just follows from the fact that maps in the Drinfeld center are, in particular, maps of $\pi$-equivariant objects in $\Coh(\mcl{G})$.  We also note that one reduces, then take invariants, in the expression $\Ext^\ast_{\Coh(\mcl{G})}(k,k)_{\rm red}^\pi$, although the order of these operations is irrelevant in terms of the topology of the spectrum.

\begin{proof}
Such surjectivity implies that the map
\[
\mbb{P}(\mcl{L}/\pi)=\Proj\Ext^\ast_{\Coh(\mcl{G})}(k,k)^\pi\to \Proj\Ext^\ast_{Z\!\Coh(\mcl{G})}(k,k)
\]
is a closed embedding, and hence that all closed subsets in $\mbb{P}(\mcl{L}/\pi)$ are realized as preimages of closed sets in $\Proj\Ext^\ast_{Z\!\Coh(\mcl{G})}(k,k)$.  One now repeats exactly the arguments in the proof of Theorem \ref{thm:center_genI} to obtain the claimed result.
\end{proof}

One instance in which such surjectivity holds is in the case in which $\mcl{G}$ admits a normal embedding into a smooth algebraic group $\mcl{H}$.  This is due to certain formality results of \cite{negron}.

\begin{corollary}\label{cor:norm}
If $\mcl{G}$ admits a normal embedding into a smooth algebraic group, then all thick ideals in $\stab(\Coh(\mcl{G}))$ are centrally generated.
\end{corollary}

\begin{proof}
In this case the deformation map $\mfk{d}:A_Z\to \Ext^\ast_{\Coh(\mcl{G})}(\1,\1)$, which is surjective modulo nilpotents \cite[(proof of) Lemma 7.5]{negronpevtsova}, is $\mcl{G}$-equivariant, where we give $A_Z$ the trivial action and $\Ext^\ast_{\Coh(\mcl{G})}(\1,\1)$ the adjoint action.  The point here is that the parametrizing subalgebra $Z=\O(\mcl{H}/\mcl{G})$ for the deformation $\mcl{H}\to \mcl{H}/\mcl{G}$ associated to a normal embedding $\mcl{G}\to \mcl{H}$ is a $\mcl{G}$-invariant subalgebra in $U=\O(\mcl{H})$.  Hence the algebra $A_Z$ acts on $D^b(\Coh(\mcl{G}))$ by $\mcl{G}$-invariant transformations in this case.
\par

By \cite[Lemma 5.2, Theorem 5.4]{negron}, and the identification
\[
\RHom_{Z\!\Coh(\mcl{G})}(\1,\1)=\RHom_\mcl{G}(\1,\RHom_{\Coh(\mcl{G})}(\1,\1))
\]
\cite[\S 7.1]{negron}, the map $\mfk{d}$ admits a lift $\tilde{\mfk{d}}:A_Z\to \Ext^\ast_{Z\!\Coh(\mcl{G})}(\1,\1)$ to the cohomology of the Drinfeld center.  It follows that the map $\Ext^\ast_{Z\!\Coh(\mcl{G})}(\1,\1)\to \Ext^\ast_{\Coh(\mcl{G})}(\1,\1)$ is surjective modulo nilpotents, and in particular induces a surjection onto the reduced $\pi$-invariant cohomology.
\end{proof}

An easy example where we have such a normal embedding $\mcl{G}\to \mcl{H}$ is the case of a Frobenius kernel $\mcl{G}=\mbb{G}_{(r)}$ in a smooth algebraic group $\mbb{G}$.  In this case we can simply take $\mcl{H}=\mbb{G}$ and observe the normal embedding $\mbb{G}_{(r)}\to \mbb{G}$.  Note that, although the classification of thick ideals in the stable category of sheaves for such $\mbb{G}_{(r)}$ is indeed known \cite{stevenson13}, it is not clear how to deduce the above central generation result from earlier works on the subject.  We record this specific finding.

\begin{corollary}
For any Frobenius kernel $\mbb{G}_{(r)}$ in a smooth algebraic group $\mbb{G}$, all thick ideals in the stable category $\stab(\Coh(\mbb{G}_{(r)}))$ are centrally generated.
\end{corollary}

We provide some additional examples which are not covered by Corollary \ref{cor:norm}.

\begin{example}
Consider $\mbb{G}$ any smooth algebraic group, and the finite group scheme
\[
\mcl{G}=(\mbb{G}_{(r)}^n)\rtimes S_n,
\]
where the symmetric group acts by permutation.  We consider the closed embedding $\mcl{G}_o=\mbb{G}_{(r)}^n\to \mbb{G}^n$ and corresponding integration.  In this case one has
\[
\Ext^\ast_{\Coh(\mcl{G})}(k,k)=A_Z\ot \wedge^\ast(\mfk{g}^\ast)=\Sym((\mfk{g}^\ast)^{(r)})\ot\wedge^\ast(\mfk{g}^\ast)
\]
as a $\pi=S_n$-algebra, where $\mfk{g}=\mrm{Lie}(\mbb{G}^n)=\mrm{Lie}(\mbb{G})^{\times n}$ and $S_n$ acts by permutation.  By the equivariant formality result of \cite[\S 5.2]{negron}, we have that the image of $Z(\Coh(\mcl{G}))$-extensions in the reduced cohomology $\Ext^\ast_{\Coh(\mcl{G})}(k,k)_{\rm red}$ is precisely
\[
\Ext^\ast_{\Coh(\mcl{G})}(k,k)^{S_n}_{\rm red}=A_Z^{S_n}.
\]
We therefore apply Theorem \ref{thm:center_genII} to observe central generation of all ideals in the stable category.
\end{example}

\begin{example}
Consider $\mcl{G}_o=\mbb{G}_{a(1)}^2$, where $\mbb{G}_a$ is the additive group, and consider also the cyclic group $\mbb{Z}/p\mbb{Z}$ with chosen generator $\sigma$.  This group scheme has algebra of functions $\O(\mcl{G}_o)=k[t_1,t_2]/(t_1^p,t_2^p)$.  We consider the action of $\mbb{Z}/p\mbb{Z}$ on $\mbb{G}_{a(1)}^2$ defined by
\[
\sigma(t_1)=t_1+t_2,\ \ \sigma(t_2)=t_2,
\]
and the smash product $\mcl{G}=\mcl{G}_o\rtimes \mbb{Z}/p\mbb{Z}$.  As with the previous example, this action extends to the ambient group $\mbb{G}_{a}^2$.  We therefore obtain an equivariant inclusion
\begin{equation}\label{eq:1531}
A_Z\to \Ext^\ast_{\Coh(\mcl{G})}(k,k)
\end{equation}
which reduces to an identification $\Ext^\ast_{\Coh(\mcl{G})}(k,k)^{\mbb{Z}/p\mbb{Z}}_{\rm red}=A_Z^{\mbb{Z}/p\mbb{Z}}$.  Since the map \eqref{eq:1531} factors through the cohomology of the center $Z(\Coh(\mcl{G}))$ \cite[\S 5.2/7.1]{negron}, we verify the hypotheses of Theorem \ref{thm:center_genII} and observe central generation of all ideals. 
\end{example}

At this point, every example which is amenable to direct calculation can be shown to satisfy the hypotheses of Theorem \ref{thm:center_genII}, and can therefore be shown to have all thick ideals in the stable category centrally generated.  We provide a more tractable version of Question \ref{q:quest_gen}.

\begin{question}\label{q:quest_G}
For an arbitrary finite group scheme $\mcl{G}$, are all ideals in $\stab(\Coh(\mcl{G}))$ centrally generated?
\end{question}

We would guess that the answer to this question is yes.  We expect that the methods which would be employed in a resolution of Question \ref{q:quest_G} may be as interesting as the resolution itself.

\subsection{A comment on the sheaf of rings}
\label{sect:not_lr}

Although we have not discussed the topic at length, let us make one comment about the sheaf of rings on prime spectra in the non-braided setting.  We expect that the sheaf of rings on the spectrum $\cSpec(\Coh(\mcl{G}))\cong_{homeo} \mbb{P}(\mcl{L}/\pi)$ should explicitly be the pushforward of the structure sheaf $\O_{\mbb{P}(\mcl{L})}$ along the projection $p:\mbb{P}(\mcl{L})\to \mbb{P}(\mcl{L}/\pi)$.  So the germs of this ringed space should be semi-local, with $|\pi|$-many simples generically, not local.
\par

Of course, it is well-established that the spectrum of a tensor triangulated category $\mcl{T}$ in the braided setting \emph{is} a locally ringed space.  However, as can be seen in the proof at \cite[Theorem 4.5]{balmer10}, one uses the fact that the quotient $\mcl{T}/\mcl{P}$ by a prime admits no zero divisors in this setting.  In the non-braided setting, products of non-zero objects in such a quotient by a thick prime can still vanish.  So in principle (and we're suggesting in actuality) the spectrum of a non-braided category can be non-locally ringed.

\section{Closing remarks on one-sided ideals}
\label{sect:onesided}

In this work we have not considered one-sided ideals.  This omission is intentional, as ``spectra" of one-sided ideals are rather disorderly individuals in general, and support seems to be better suited to the classification of two-sided ideals.  However, in the case of $\Coh(\mcl{G})$, one observes some intriguing phenomena with respect to one-sided ideals in the stable category.
\par

Consider $\mcl{G}$ a (generally non-connected) finite group scheme.  First, we note that cohomological support $\supp^{coh}_\msf{Y}$ for $\Coh(\mcl{G})$ produces one-sided ideals in the stable category.  Specifically, for any specialization closed subset $\Theta\subset \msf{Y}$ one produces a corresponding ideal
\[
\msc{M}_\Theta:=\{V\in \stab(\Coh(\mcl{G})):\supp_\msf{Y}^{coh}(V)\subset \Theta\}.
\] 
This assignment provides an \emph{injective} map
\[
\begin{array}{l}
\msc{M}_?:\\
\{\text{Specialization closed subsets in }\msf{Y}\}\to \{\text{Thick one-sided ideals in }\stab(\Coh(\mcl{G}))\}.
\end{array}
\]
We ask the following.

\begin{question}
Is the above map $\msc{M}_?$ a bijection.  That is to say, are one-sided ideals in $\stab(\Coh(\mcl{G}))$ classified by cohomological support?  More to the point, does cohomological support actually classify \emph{anything} for $\Coh(\mcl{G})$, at general $\mcl{G}$?
\end{question}

As remarked above, cohomological support for $\Coh(\mcl{G})$, which is the same as hypersurface support, is not stable under duality in this case
\[
\supp^{coh}_\mbb{P}(V)\neq \supp^{coh}_\mbb{P}(V^\ast).
\]
This failure of support to be stable under duality obstructs attempts to deal with this classification problem via local cohomology functors.  Specifically, the methods employed in the proof of Proposition \ref{prop:classify} break down, due to these incompatibilities between support and duality.
\par

One can ask a similar question about one-sided ideals for the stable category of representations for a quantum complete intersection, in which case support \emph{is} stable under duality.  So, in this case one might believe that one-sided ideals are (also) classified by cohomological support.  However, one should ask themselves at this point if we should even \emph{expect} such phenomena to occur in general.  That is to say, should we expect that (a generic version of) the map $\msc{M}_?$ classifies one-sided ideals in the stable categories $\stab(\msf{u})$ for general $\msf{u}$.  In order to investigate this question, the more extreme example of coherent sheaves on a non-connected group scheme $\mcl{G}$ seems to provide the greater possibility of insight.

\renewcommand{\O}{\oldO}
\bibliographystyle{abbrv}

\end{document}